\def\newversion{1} 
\def\diffcolor{black} 

\newcommand{\towriteornottowrite}[2]{\ifthenelse{\equal{\newversion}{2}}{{\color{Green}#1\color{red}#2}}{\ifthenelse{\equal{\newversion}{1}}{{\color{\diffcolor}#1}}{{\color{\diffcolor}#2}}}}

\documentclass{article}
\usepackage[utf8]{inputenc}
\usepackage[T1]{fontenc}

\usepackage{rotating}
\usepackage{amssymb}
\usepackage{amsthm}
\usepackage{amsmath}
\usepackage{tabularx}
\usepackage{verbatim}
\usepackage[linesnumbered, ruled, boxed, vlined]{algorithm2e}
\usepackage{enumitem}
\usepackage{url}
\usepackage{float}
\usepackage{mathtools}
\usepackage{array}
\usepackage{hyperref}

\newcommand{\layout}{\textsc{layout}}
\newcommand{\labels}{\textsc{labels}}

\usepackage[dvipsnames]{xcolor}
\definecolor{mygreen}{rgb}{0.01, 0.75, 0.24}
\usepackage{tikz}
\def\mycolor{black}

\newcommand{\newdot}[2]{
 \node[shape=circle, minimum size = 5pt, inner sep=0pt, outer sep=0pt, fill=\mycolor, color=\mycolor] (#1) at (#2){};
}

\usepackage[a4paper, margin=2.5cm,top=3cm, bottom=3cm]{geometry}

\setlist[enumerate,1]{label={\textnormal{(\roman*)}}}

\theoremstyle{plain}
\newtheorem{theorem}{Theorem}[section]
\newtheorem{lemma}[theorem]{Lemma}
\newtheorem{proposition}[theorem]{Proposition}

\newtheorem{conjecture}[theorem]{Conjecture}
\newtheorem{observation}[theorem]{Observation}
\newtheorem{question}[theorem]{Question}

\theoremstyle{definition}
\newtheorem{definition}[theorem]{Definition}

\newcommand{\oldqed}{}
\def\endofClaim{\hfill\scalebox{.6}{$\Box$}}

\def\textcol{black}
\newcommand{\textbe}[2]{
\node[label=below:{\textcolor{\textcol}{#2}}] at (#1) {};
}
\newcommand{\textab}[2]{
\node[label=above:{\textcolor{\textcol}{#2}}] at (#1) {};
}
\newcommand{\textat}[2]{
\node at (#1) {\textcolor{\textcol}{#2}};
}
\usepackage{savesym}
\savesymbol{textle}
\newcommand{\textle}[2]{
\node[label=left:{\textcolor{\textcol}{#2}}] at (#1) {};
}
\newcommand{\textri}[2]{
\node[label=right:{\textcolor{\textcol}{#2}}] at (#1) {};
}
\renewcommand{\hat}[1]{\widehat{#1}}
\newcommand{\upper}[1]{\overset{\textcolor{blue}{\hspace{.1cm}#1\hspace{.1cm}}}{~}}
\newcommand{\domath}[2]{\pgfmathsetmacro{#1}{#2}}
\newcommand{\floor}[1]{\left\lfloor#1\right\rfloor}

\newcommand\eps{\varepsilon}

\newcommand{\coleq}{\coloneqq}

\newcommand\cF{\mathcal{F}}

\newcommand{\F}{\mathbb F}
\newcommand{\Z}{\mathbb Z}
\newcommand{\C}{\mathcal C}
\newcommand{\spa}[1]{\left\langle#1\right\rangle}
\newcommand{\subs}{\subseteq}
\newcommand{\abs}[1]{\left\lvert#1\right\rvert}

\DeclareMathOperator{\ord}{ord}

\author{Olaf Parczyk\thanks{Fachbereich Mathematik und Informatik, Freie Universität Berlin, Arnimallee 3, 14195 Berlin, Germany.\\ \textit{E-mails:} \texttt{parczyk|szabo@mi.fu-berlin.de, s.rathke@fu-berlin.de}.}, Silas Rathke\footnotemark[1], and Tibor Szabó\footnotemark[1] }

\date{}

\title{The maximum diameter of $2$-dimensional simplicial complexes}

\begin{document}

\maketitle

\begin{abstract}
    We study a problem of Santos about the largest possible diameter of a $d$-dimensional (abstract) simplicial complex on $n$ vertices.
    For dimension $2$, we determine the exact value of the maximum for every $n$ using an explicit construction.
    We also come across a tantalizing open problem about the packing of squares of Hamilton cycles in the complete graph and obtain an infinite sequence of tight explicit constructions.
    \end{abstract}

\section{Introduction}

A \emph{simplicial complex} on $n$ vertices is a family $\cF$ of subsets of $[n]$, which is closed under taking subsets.
If the maximum size of a set in $\cF$ is $d+1$, then it is called a \emph{simplicial $d$-complex}.
The \emph{dual graph} of a simplicial $d$-complex $\cF$ has as vertices all sets in $\cF$ of size $d+1$ and two vertices are adjacent if their intersection has size $d$.
The \emph{diameter} of a simplicial complex $\cF$ is the diameter of its dual graph.

Let $H_s(n,d)$ be the maximum diameter of a simplicial $d$-complex on $n$ vertices with a connected dual graph. This notion was introduced by Santos~\cite{santosderzweite} while studying possible generalizations of the polynomial Hirsch conjecture. He proved the bounds 
\begin{equation}\label{eq:santosbounds} \Omega\!\left(n^{\frac{2d+2}{3}}\right)\le H_s(n,d) \le \frac{1}{d}\binom nd.
\end{equation} 
Subsequently, Criado and Santos~\cite{criado2017maximum} gave an explicit construction of simplicial $d$-complexes whose diameter matched the order of magnitude of the upper bound for every fixed $d$ and an infinite sequence of $n$.
Criado and Newman~\cite{criado2021randomized} used the Lovász Local Lemma to establish the existence of a construction for any fixed $d\ge 3$ and all $n$, that reduced the gap between the bounds from a factor exponential in $d$ to $\mathcal O(d^2)$.
Most recently Bohman and Newman~\cite{bohman2022complexes} managed to pin down the precise asymptotics using the differential equations method to track the evolution of a random greedy algorithm: \begin{equation}\label{eq:bohmanbound}  
\left( \frac{1}{d} - (\log n)^{-\varepsilon} \right) \binom{n}{d} \le H_s(n,d) \, ,
\end{equation}
where $\eps < 1/d^2$ and $n$ is sufficiently large.

The upper bound in~(\ref{eq:santosbounds}) is in fact quickly justified. On a shortest path \towriteornottowrite{$P$}{} between any two vertices in the dual graph, each vertex (corresponding to a $(d+1)$-set) contains $d$ sets of size $d$ that are not contained in any previously visited vertex as this would create a shortcut. Hence, if $\ell$ is the number of vertices of \towriteornottowrite{$P$}{the path}, then there must be at least $d\cdot \ell$ sets of size $d$. Thus, $\ell\le\frac{1}{d}\binom{n}{d}$ and the longest diameter can be at most $\frac{1}{d}\binom{n}{d}-1$. Since the first vertex of $P$ actually covers $d+1$ sets of size $d$, we get the slightly improved upper bound
\begin{equation}\label{eq:upper bound}
H_s(n,d)\le \floor{\frac{1}{d}\binom{n}{d}-\frac{d+1}{d}}.
\end{equation}

The main result of this paper is a matching lower bound for $d=2$ for every $n$ not equal to 6.

\begin{theorem}
    For $n\ge 3$, it holds that 
    \label{thm:main, 4k+1}
    \[H_s(n,2) = \begin{cases}
        \floor{\frac{1}{2}\binom{n}{2}-\frac{3}{2}} & n\ne 6\\
        5=\floor{\frac{1}{2}\binom{6}{2}-\frac{3}{2}}-1 & n=6.
    \end{cases}\]
\end{theorem}
We obtain these values by explicitly constructing simplicial $2$-complexes that have this diameter. It is worth noting that the previously best known lower bound \cite{bohman2022complexes},  with error term of order $d^2/(\log d)^\varepsilon$, was obtained by random methods.

One way to obtain a simplicial $2$-complex with large diameter, formulated in graph theoretic language, is to build a long \emph{trail-square}, that is, a sequence of (not necessarily distinct) vertices such that every pair of vertices occurs at most once within distance two in the sequence.
For example, any ordering of the vertex set provides such a sequence, and the pairs appearing at distance at most two in the sequence form the square\towriteornottowrite{\footnote{The \emph{square} $G^{(2)}$ of a graph $G$ is the graph defined on the vertex set $V(G)$ by connecting the pairs of vertices whose distance in $G$ is at most two.}}{} of a Hamilton path.
But this is only a linear number of pairs however, while in an optimal trail-square we are aiming to cover (almost) all the pairs. 
Our initial approach to this problem was to create a large family of pairwise edge-disjoint squares of Hamilton paths and ``glue’’ them up into a trail-square with the introduction of a new vertex or two at each connection.
How much these new vertices can be reused at more than one connection might depend on the ends of the squares of Hamilton paths being ``well-distributed''.
To gain flexibility in that regard, a decomposition of the edges of the complete graph into squares of Hamilton cycles proves to be quite handy, since each can then be ``cut’’ anywhere to obtain a family of pairwise edge-disjoint squares of Hamilton paths with a good variety of endings.
This leads us to the following natural\towriteornottowrite{, purely graph-theoretical }{ }question:

\begin{question}
\label{quest:HC2}
For which positive integers $n$ can the edge set of $K_n$ be partitioned into squares of Hamilton cycles?
\end{question}

\towriteornottowrite{Unless $n=3$ or $4$, the square of a Hamilton cycle is 4-regular. Consequently for $n\geq 5$, a partition of $E(K_n)$ into squares of Hamilton cycles can only exist if each vertex degree is divisible by 4, i.e.\@ $n \equiv 1 \pmod 4$.}{
 In case a partition of $E(K_n)$ into squares of Hamilton cycles exists, one certainly needs that each vertex has a degree that is divisible by $4$. Hence $n$ being congruent to $1 \pmod 4$ is a necessary condition.}
 This is also conjectured to be sufficient, at least for large enough $n$, as a special case of a general conjecture about packing bounded degree graphs into the complete graph~\cite[Conjecture~1.4]{Oberwolfach}.

The answer to the analogous question about Hamilton cycles has been known since at least 1892.
A construction, which exists for all odd $n$, is attributed to Walecki (see \cite{Walecki} and \cite{alspach2008wonderful}).
The work of Ferber, Lee, and Mousset~\cite{FerberLeeMousset} proves the existence of families of pairwise edge-disjoint squares of Hamilton cycles which cover all but $o(n^2)$ edges of $K_n$.

Here we exhibit an explicit construction of a perfect decomposition of the edges of the complete graph into squares of Hamilton cycles for an infinite sequence of $n$. 

\begin{theorem}
\label{thm:packHC2}
    Let $p \equiv 1 \pmod 4$ be a prime number, such that the order $\ord_p(2)$ of $2$ modulo $p$ is divisible by four. Then $E(K_p)$ can be partitioned into pairwise edge-disjoint copies of $C_p^2$. This is the case in particular if $p\equiv 5 \pmod 8$.
\end{theorem}

For our application, this theorem is more advantageous to use than the unspecified $o(n^2)$ error term of~\cite{FerberLeeMousset}. This is because for any $n$ there is a prime $p \in [n-n^{0.53},n]$ which satisfies $p \equiv 5 \pmod 8$,~\cite{Kumchev_consecutive_arithmetic}, so we can cover all but $\mathcal O(n^{1.53})$ edges of $K_n$ by squares of long cycles. Cutting these appropriately and connecting them up into a long trail-square with the introduction of a couple of new vertices in between, one can make sure that only $n^{1.53}$ edges of $K_n$ are not being covered. This way we end up with a construction of diameter $\frac{1}{2} \binom{n}{2} - \mathcal O(n^{1.53})$, already improving the result from~\cite{bohman2022complexes} for $d=2$. 

In relation to Theorem~\ref{thm:main, 4k+1}, for $n=p$ from Theorem~\ref{thm:packHC2} one can make the gluing described above a bit more efficient, so it leaves out only a linear number of edges. 
It turns out however that this can not be improved in general if we insist on a trail-square construction. 
Indeed, if a simplicial $2$-complex is constructed in the trail-square manner, then for any $i\in [n]$ which is not among the first or last two elements of the sequence, the number of $2$-sets of the simplicial complex containing $i$ is a multiple of four. 
Thus, when restricted to vertex-sequential, there is an infinite sequence of $n$, those where $n-1$ is not divisible by four, where the error term will be of order $\Omega (n)$.

Since in Theorem \ref{thm:main, 4k+1}, we want to match the upper bound for all $n$, we have to do it differently and consider simplicial complexes that are not of the vertex-sequential form.
This will be introduced in Section~\ref{sec:Thm1.1}. 
The proof of Theorem~\ref{thm:packHC2} is given in Sections~\ref{sec:proofHC2}.
In the concluding remarks, Section~\ref{sec:conclusion}, we mention a generalization of the ``partition/cut/glue'' approach to $d \ge 3$ and contemplate about a couple of open problems/conjectures, in particular, continue the discussion of Question~\ref{quest:HC2}.

\section{Proof of Theorem~\ref{thm:main, 4k+1}} \label{sec:Thm1.1}

This section is structured as follows.
First, in Section~\ref{sec:visualize}, we will discuss how the simplicial 2-complexes of interest can be encoded as a sequence of vertices together with a $0/1$-sequence.
Afterwards, Section~\ref{sec:gen sequences} will introduce the concept of generating sequences which is the main ingredient for constructing simplicial 2-complexes in a periodic way.
These generating sequences will already be sufficient to find simplicial 2-complexes with optimal diameter when $n\equiv 1\pmod 4$.
This is done in Section ~\ref{sec:4k+1}.

The other three residues are dealt with in the subsequent three subsections.
In all of these cases, we will find a suitable generating sequence for some $n' \equiv 1 \pmod 4$ where $n'$ is only a constant smaller than $n$. This together with an ad-hoc construction for the remaining vertices is sufficient to find a simplicial 2-complex whose diameter is only an additive constant away from the maximum possible diameter.
However, since we want to find simplicial 2-complexes with optimal diameter, we need to carefully modify these constructions to accommodate every edge.

Unsurprisingly, these constructions and the suitable generating sequences only exist when $n$ is large enough.
Fortunately, $n> 30$ is already large enough.
In Appendix~\ref{sec:appendix}, we explicitly give optimal simplicial 2-complexes for the values of $n\le30$. This will then conclude the proof of Theorem \ref{thm:main, 4k+1}.

\subsection{Encoding simplicial 2-complexes}\label{sec:visualize}

In this section, we will just say simplicial complex for a simplicial $2$-complex and sets of size two will be called edges.
As mentioned above, a simplicial complex of diameter $t$ contains a sequence of $t+1$ sets of size three such that two consecutive sets intersect in two vertices and any other pair intersects in at most one vertex.
We call such a sequence of sets of size three \emph{good} and also consider it as a simplicial complex.
If the first and last set also intersect in two vertices, we call the sequence \emph{circular} and also refer to the sequence as \emph{good}.
Such a circular sequence, by abuse of notation, we also call it a \emph{circular simplicial complex}.

We can visualize a good sequence as a sequence of triangles where neighboring triangles share an edge and each vertex is labeled\footnote{Later we will work with residues and elements in $\Z/n\Z$. Thus, we chose the labels to be in $\lbrace 0,\dots, n-1\rbrace$ instead of $\lbrace1,\dots,n\rbrace$.} with an element from $\lbrace 0,\dots n-1\rbrace$ such that every pair of labels appears at most once as an edge.
Figure~\ref{fig:simplicial complex example} gives an example of a simplicial 2-complex for $n=7$ of diameter 7.

\begin{figure}[H]
\centering
\begin{tikzpicture}
    \newdot{a1}{-4,0}
    \newdot{a2}{-2,0}
    \newdot{a4}{2,0}
    {\def\mycolor{red}\newdot{a6}{6,0}}
    \newdot{b1}{-3,1.7321}
    \newdot{b2}{-1,1.7321}
    \newdot{b4}{3,1.7321}
    \newdot{a3}{0,0}
    \newdot{a5}{4,0}
    \newdot{b3}{1,1.7321}
    \newdot{b5}{5,1.7321}
    \draw(a1) -- (a3) -- (b2) -- (b1) -- (a1) (b1) -- (a2) -- (b2);
    \draw(b5) -- (b4) -- (a5) (b5) -- (a5) -- (b4);
    \draw(b2) -- (b4) -- (a4) -- (b3) -- (a3) -- (a5);
    \draw[red](a5) -- (a6) -- (b5);
    \textbe{a1}{0}
    \textbe{a2}{2}
    \textbe{a3}{4}
    \textbe{a4}{5}
    \textbe{a5}{1}
    \textab{b1}{1}
    \textab{b2}{3}
    \textab{b3}{0}
    \textab{b4}{6}
    \textab{b5}{4}
    \def\textcol{red}
    \textbe{a6}{?}
\end{tikzpicture}
\caption{The black sequence of triangles represents a simplicial 2-complex of diameter 7 for $n=7$. We cannot add the red triangle to it to create a simplicial 2-complex of diameter 8 because all edges containing $1$ or $4$ are already used.}
\label{fig:simplicial complex example}
\end{figure}
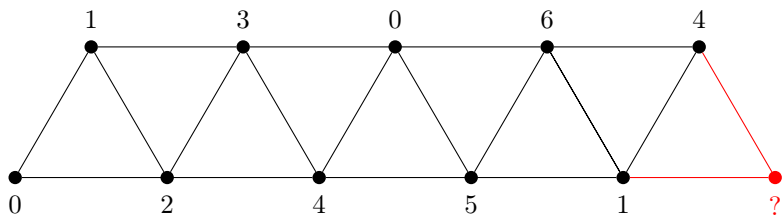

If we add the red triangle, then no matter what label the red vertex gets, an edge will appear twice in the simplicial complex, which creates a shortcut in the dual graph and the corresponding sequence of sets of size three would not be good.
In a simplicial complex, the triangles do not have to be in a straight line, but we can also not add a triangle containing $4$ and $6$ in Figure~\ref{fig:simplicial complex example}, since again this would create a shortcut.
An example of a simplicial complex where the triangles are not in a straight line is shown in Figure~\ref{n=5 and n=9}.
Hence, to encode a simplicial complex we not only have to specify the labels of the vertices but also the layout of the triangles.

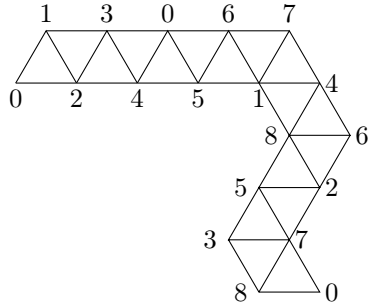
\begin{figure}[H]
\centering
\begin{tikzpicture}[scale=0.8]
        \begin{scope}[shift={(0,0)}]
        \foreach \x in {0,...,4} {
        \begin{scope}[shift={(\x,0)}]
        \draw (0,0) -- (0:1);
        \draw[] (0:1) -- (60:1) -- (0,0);
        \end{scope}
        }
        \draw (0.5,0.86603) -- (4.5,0.86603) (4,0) -- (4.5,-0.866) -- (5,0) -- (5.5,-.866) -- (4.5,-0.866) -- (5,-1.732) -- (5.5,-.866) (4.5,-.866) -- (4,-1.732) -- (5, -1.732) -- (4.5,-2.598) -- (3.5,-2.598) -- (4,-1.732) -- (4.5,-2.598) -- (4,-3.464)  (3.5,-2.598) -- (4,-3.464) -- (5,-3.464) -- (4.5,-2.598);
        \textbe{0,.2}{0}
        \textbe{1,.2}{2}
        \textbe{2,.2}{4}
        \textbe{3,.2}{5}
        \textbe{4,.2}{1}
        \textat{5.2,0}{4}
        \textat{4.2,-0.866}{8}
        \textat{3.7,-1.732}{5}
        \textat{3.2,-2.598}{3}
        \textat{5.7,-0.866}{6}
        \textat{5.2,-1.73}{2}
        \textat{4.7,-2.598}{7}
        \textat{5.2,-3.464}{0}
        \textat{3.7,-3.464}{8}
        \textab{.5,.7}{1}
        \textab{1.5,.7}{3}
        \textab{2.5,0.7}{0}
        \textab{3.5,0.7}{6}
        \textab{4.5,0.7}{7}
        \end{scope}
\end{tikzpicture}
\caption{Simplicial complex for $n=9$.}
\label{n=5 and n=9}
\end{figure}

We encode the sequence of sets of size three with a pair of sequences (\labels{}, \layout{}) of length $t+3$ and $t$ respectively.
The sequence \labels{} is of the form $[x_0,x_1,\dots,x_{t+2}]$ with $x_i \in \lbrace 0,\dots,n-1\rbrace$ for $i=0,\dots,t+2$ and describes how the vertices of the triangles will be labelled.
The sequence \layout{} is of the form $[y_3,\dots,y_{t+2}]$ with $y_i \in \lbrace 0,1 \rbrace$ for $i=3,\dots,t+2$ and encodes at which edge the next triangle is appended.

From these two sequences the \emph{corresponding} sequence of sets of size three and, thereby, the \emph{corresponding} simplicial complex, is constructed \towriteornottowrite{set by set}{1-by-1} as follows:
The first set is $\{x_0,x_1,x_2\}$.
\towriteornottowrite{Then, given a sequence of $i-2$ sets of size three for some $i=3, \ldots, t+2$, such that the last set is $\{x_j,x_{i-2},x_{i-1}\}$ for some\footnote{Even though it is not necessary for the further arguments, we note that $j$ equals $i-3-\ell$ where $\ell$ is the number of consecutive 1's the partial \layout{} vector $[y_3,\dots,y_{i-1}]$ ends in.} $j<i-2$, the next set is defined to be $\{ x_{i-2},x_{i-1},x_i \}$ if the \layout{} entry $y_i=0$ and $\{ x_j,x_{i-1},x_i \}$ if $y_i=1$.}{Then for $i=3,\dots,t+2$ and last set $\{x_j,x_{i-2},x_{i-1}\}$, $j<i-2$, the next set is $\{ x_{i-2},x_{i-1},x_i \}$ if $y_i=0$ and $\{ x_j,x_{i-1},x_i \}$ if $y_i=1$.}
The other way around we can construct the two sequences from a good sequence of $t+1$ sets of size three:
We start at either end of the sequence and let $\{ x_0,x_1,x_2 \}$ be the first set, where $x_1,x_2$ is also contained in the next set.
Then for $i=3,\dots,t+2$ we pick $x_i$ such that the $(i-1)$-st set is $\{ x_j,x_{i-1},x_i \}$, $j<i-1$, and let $y_i=0$ if $x_j=x_{i-2}$ and $y_i=1$, otherwise.
For instance, the simplicial complex in Figure \ref{n=5 and n=9} can be described via
\begin{align*}
    \labels:&&[0,1,2,{}&3,4,0,5,6,1,7,4,8,6,2,5,7,3,8,0]\\
    \layout:&&&\hspace{-.1cm}[0,0,0,0,0,0,0,0,1,0,0,1,0,0,0,1].
\end{align*}

We call the pair of sequences $(\labels{},\layout{})$ \emph{good} if the corresponding sequences of sets of size three are good.
This includes the case that the first two and the last two entries of \labels{} are the same, in which case the simplicial complex is \emph{circular}.
For any edge that is contained in a set of size three in the simplicial complex, we say that this edge is \emph{covered}.
Clearly, the pair of sequences $([x_0,x_1,\dots,x_{t+2}],[y_3,\dots,y_{t+2}])$ is good if and only if the corresponding simplicial complex covers $2t+3$ ($2t+2$ if it is circular) edges.
The following observation summarises the main connection between the sequence of sets of size three and the diameter of the corresponding simplicial complex.

\begin{observation}
    \label{obs:sequences}
    If a sequence of sets of size three is good and covers $2t+3$ edges, then it has diameter $t$ as a simplicial complex.
\end{observation}

The goal is to create a good sequence of sets of size three which is as long as possible and, thereby, covers as many edges as possible.
The previous observation immediately implies the upper bound $\floor{\frac{1}{2}\binom{n}{2}-\frac{3}{2}}$ on the diameter of a simplicial complex.
We will show that for $n\ne6$, it is always possible to reach this upper bound.
If $\binom{n}{2}$ is odd, then every edge must be covered, while
if $\binom{n}{2}$ is even, there will be exactly one edge that is not covered.
For example, in Figure~\ref{n=5 and n=9} all edges are covered except for $\{3,6\}$.
Therefore, $n\equiv0, 1 \pmod 4$ will be a bit easier compared to $n\equiv2,3 \pmod 4$.

\subsection{Generating sequences and circular simplicial complexes}\label{sec:gen sequences}

To construct the simplicial complexes we restrict ourselves to those, where the difference of neighboring elements in the \labels{} sequence is periodic.
This will allow us to generate the sequence from a significantly shorter sequence.
For this assume $n=4k+1$ for some integer $k$ and let the $n$ elements be those from $\Z/n\Z$.
We say that the \emph{residue} of a pair of labels $\lbrace i,j\rbrace\in\binom{\Z/n\Z}{2}$ is an $r\in\Z/n\Z\backslash\lbrace 0\rbrace$ such that $\lbrace i-j,j-i\rbrace=\lbrace r,-r\rbrace$. Note that if $r$ is a residue for an edge, then $-r$ is also a residue for this edge.

Now, for example, let $n=13$ and consider the sequence of differences $(1,2,4)$.
By starting at $0$ and periodically adding $+1$, $+2$, and $+4$ in $\Z/n\Z$ we obtain a \labels{} sequence
\[ [0,1,3,7,8,10,1,2,4,8,9,11,2,3,5,9,10,12,3,4,6,10,11,0,4,5,7,11,12,1,5,6,8,12,0,2,7,9,0,1]\]
that together with the all zero \layout{} gives a circular simplicial complex, where every edge appears exactly once as shown in Figure~\ref{fig:n=13}.
Moreover, by removing any edge that is only supporting one triple and the corresponding triple, we obtain a simplicial complex of diameter $H_s(13,2)=37$.
One crucial observation from this example is that the sequence $1,2,4$ together with the pairwise sums $1+2,2+4,4+1$ exactly covers $\{ 1,2,3,4,5,6\}$, the residues of all edges.
To indicate the additional residues that are covered as well, we will later write this sequence as $1 \upper{3} 2 \upper{6} 4 \upper {5}$.
Of course, any circular permutation of the sequence gives the same construction.

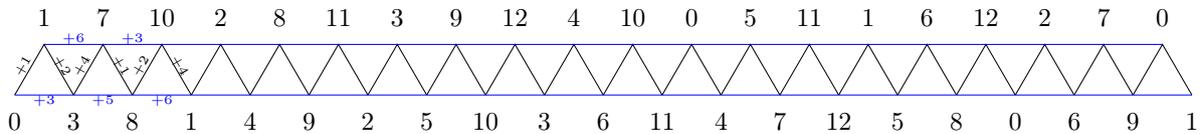
\begin{figure}[H]
    \centering
    \resizebox{\linewidth}{!}{
    \begin{tikzpicture}[scale=0.8]
        \foreach \x in {0,...,18} {
        \begin{scope}[shift={(\x,0)}]
        \draw[blue] (0,0) -- (0:1);
        \draw[black] (0:1) -- (60:1) -- (0,0);
        \end{scope}
        }
        \draw[black] (19,0) -- (19.5,0.86603) -- (20,0);
        \draw[blue] (0.5,0.86603) -- (19.5,0.86603)  (20,0) -- (19,0);
        \textbe{20,0}{1}
        \textbe{0,0}{0}
        \textbe{1,0}{3}
        \textbe{2,0}{8}
        \textbe{3,0}{1}
        \textbe{4,0}{4}
        \textbe{5,0}{9}
        \textbe{6,0}{2}
        \textbe{7,0}{5}
        \textbe{8,0}{10}
        \textbe{9,0}{3}
        \textbe{10,0}{6}
        \textbe{11,0}{11}
        \textbe{12,0}{4}
        \textbe{13,0}{7}
        \textbe{14,0}{12}
        \textbe{15,0}{5}
        \textbe{16,0}{8}
        \textbe{17,0}{0}
        \textbe{18,0}{6}
        \textbe{19,0}{9}
        \textab{0.5,0.86603}{1}
        \textab{1.5,0.86603}{7}
        \textab{2.5,0.86603}{10}
        \textab{3.5,0.86603}{2}
        \textab{4.5,0.86603}{8}
        \textab{5.5,0.86603}{11}
        \textab{6.5,0.86603}{3}
        \textab{7.5,0.86603}{9}
        \textab{8.5,0.86603}{12}
        \textab{9.5,0.86603}{4}
        \textab{10.5,0.86603}{10}
        \textab{11.5,0.86603}{0}
        \textab{12.5,0.86603}{5}
        \textab{13.5,0.86603}{11}
        \textab{14.5,0.86603}{1}
        \textab{15.5,0.86603}{6}
        \textab{16.5,0.86603}{12}
        \textab{17.5,0.86603}{2}
        \textab{18.5,0.86603}{7}
        \textab{19.5,0.86603}{0}
        \node[rotate=60] () at (0.15,0.5) {\tiny{$+1$}};
        \node[rotate=300] () at (0.8,0.5) {\tiny{$+2$}};
        \node[rotate=60] () at (1.15,0.5) {\tiny{$+4$}};
        \node[rotate=300] () at (1.8,0.5) {\tiny{$+1$}};
        \node[rotate=60] () at (2.15,0.5) {\tiny{$+2$}};
        \node[rotate=300] () at (2.8,0.5) {\tiny{$+4$}};
        \node[blue] () at (.5,-0.1) {\tiny{$+3$}};
        \node[blue] () at (1.5,-0.1) {\tiny{$+5$}};
        \node[blue] () at (2.5,-0.1) {\tiny{$+6$}};
        \node[blue] () at (1,0.96) {\tiny{$+6$}};
        \node[blue] () at (2,.96) {\tiny{$+3$}};
    \end{tikzpicture}}
    \caption{Circular simplicial complex for $n=13$. The first and the last triangles are glued together via the edge $\lbrace0,1\rbrace$. It can be cut at any blue edge to create an optimal simplicial complex.}
    \label{fig:n=13}
\end{figure}

\towriteornottowrite{Unfortunately, we were not able to find such a sequence for every $n=4k+1$, but once we allow turns, we are able to find such a sequence for all $n=4k+1$.}{Unfortunately, this does not easily extend to $n=17$, but by including turns we do find a periodic sequence that works.}
For example, \towriteornottowrite{for $n=17$ and the sequence $(3,2,7,6)$,}{from the sequence $(3,2,7,6)$} we can periodically generate the \labels{} sequence as before, but now we put a $1$ in the \layout{} sequence each time we did $+6$ in the step before. We write this as $(3,2,7,\hat 6)$ where the hat above the 6 indicates that we always put a $1$ in the \layout{} sequence each time we did $+6$ in the step before.

Again we obtain a circular simplicial complex that covers all edges and is illustrated in Figure~\ref{fig:n=17}.
As above, we also write this sequence as $3\upper52\upper 97\upper{13}\hat 6\upper{16}$ to indicate the extra edges covered, where the turn gives $7+6+3=16$ as a residue. We can see that for each $r\in(\Z/17\Z)\backslash\lbrace 0\rbrace$, either $r$ or $-r$ appears in this sequence.

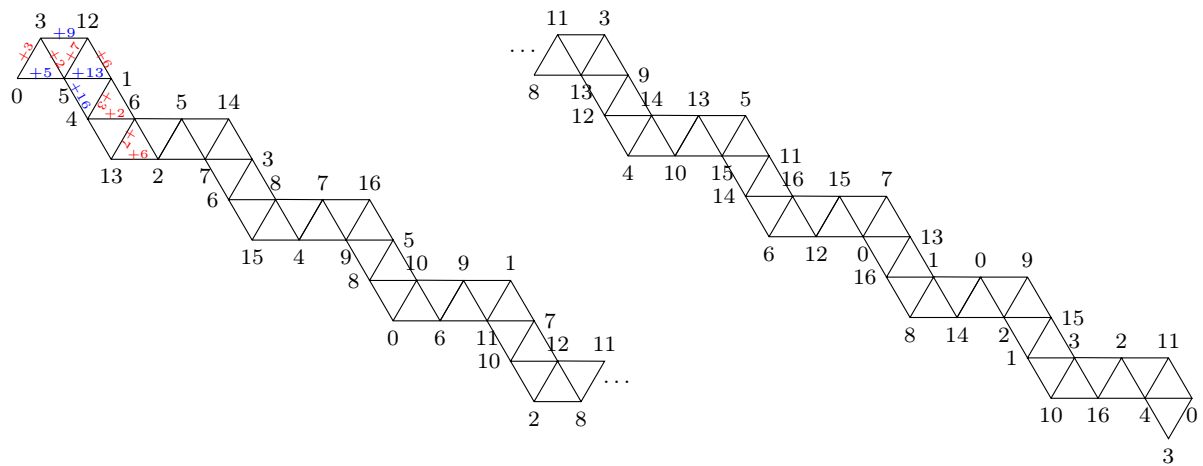
\begin{figure}[H]
    \centering
    \resizebox{\linewidth}{!}{
    \begin{tikzpicture}[scale=0.6]
    \tikzstyle{every node}=[font=\footnotesize]
        \foreach \x in {0,...,3} {
        \domath{\newx}{\x*3}
        \domath{\newy}{-\x*1.73205}
        \begin{scope}[shift={(\newx,\newy)}]
        \draw (0,0) -- (1,0) -- (2,-1.73205) -- (3,-1.73205) -- (3.5,-0.86603) -- (2.5,-0.8603) -- (1.5, 0.86603) -- (0.5,0.86603) -- (0,0) (0.5,0.86603) -- (1,0) -- (1.5,0.86603) (1,0) -- (2,0) -- (1.5,-0.86603) -- (2.5,-0.86603) -- (2,-1.73205) (3,-1.73205) -- (2.5,-0.86603);
        \end{scope}
        }

        \def\shiftx{1}
        \def\shifty{7}
        \begin{scope}[shift={(-\shiftx,\shifty)}]
        \foreach \x in {4,...,7} {
        \domath{\newx}{\x*3}
        \domath{\newy}{-\x*1.73205}
        \begin{scope}[shift={(\newx,\newy)}]
        \draw (0,0) -- (1,0) -- (2,-1.73205) -- (3,-1.73205) -- (3.5,-0.86603) -- (2.5,-0.8603) -- (1.5, 0.86603) -- (0.5,0.86603) -- (0,0) (0.5,0.86603) -- (1,0) -- (1.5,0.86603) (1,0) -- (2,0) -- (1.5,-0.86603) -- (2.5,-0.86603) -- (2,-1.73205) (3,-1.73205) -- (2.5,-0.86603);
        \end{scope}
        }
        \end{scope}

        \textbe{0,0.2}{0}
        \textab{0.5,0.66603}{3}
        \textbe{1,0.2}{5}
        \textab{1.5,0.66603}{12}
        \textri{1.8,0}{1}
        \textle{1.7,-0.86603}{4}
        \textab{2.5,-1.06603}{6}
        \textbe{2,-1.53205}{13}
        
        \domath{\newx}{3}
        \domath{\newy}{-1.73205}
        \begin{scope}[shift={(\newx,\newy)}]
        \textbe{0,0.2}{2}
        \textab{0.5,0.66603}{5}
        \textbe{1,0.2}{7}
        \textab{1.5,0.66603}{14}
        \textri{1.8,0}{3}
        \textle{1.7,-0.86603}{6}
        \textab{2.5,-1.06603}{8}
        \textbe{2,-1.53205}{15}
        \end{scope}

        \domath{\newx}{2*3 }
        \domath{\newy}{2*-1.73205}
        \begin{scope}[shift={(\newx,\newy)}]
        \textbe{0,0.2}{4}
        \textab{0.5,0.66603}{7}
        \textbe{1,0.2}{9}
        \textab{1.5,0.66603}{16}
        \textri{1.8,0}{5}
        \textle{1.7,-0.86603}{8}
        \textab{2.5,-1.06603}{10}
        \textbe{2,-1.53205}{0}
        \end{scope}

        \domath{\newx}{3*3}
        \domath{\newy}{3*-1.73205}
        \begin{scope}[shift={(\newx,\newy)}]
        \textbe{0,0.2}{6}
        \textab{0.5,0.66603}{9}
        \textbe{1,0.2}{11}
        \textab{1.5,0.66603}{1}
        \textri{1.8,0}{7}
        \textle{1.7,-0.86603}{10}
        \textab{2.5,-1.06603}{12}
        \textbe{2,-1.53205}{2}
        \textbe{3,-1.53205}{8}
        \textab{3.5,-1.06602}{11}
        \end{scope}

        \textat{12.8,-6.5}{$\dots$}
        \textat{10.8,.6}{$\dots$}

        \domath{\newx}{4*3 - \shiftx}
        \domath{\newy}{4*-1.73205 + \shifty}
        \begin{scope}[shift={(\newx,\newy)}]
        \textbe{0,0.2}{8}
        \textab{0.5,0.66603}{11}
        \textbe{1,0.2}{13}
        \textab{1.5,0.66603}{3}
        \textri{1.8,0}{9}
        \textle{1.7,-0.86603}{12}
        \textab{2.5,-1.06603}{14}
        \textbe{2,-1.53205}{4}
        \end{scope}

        \domath{\newx}{5*3 - \shiftx}
        \domath{\newy}{5*-1.73205 + \shifty}
        \begin{scope}[shift={(\newx,\newy)}]
        \textbe{0,0.2}{10}
        \textab{0.5,0.66603}{13}
        \textbe{1,0.2}{15}
        \textab{1.5,0.66603}{5}
        \textri{1.8,0}{11}
        \textle{1.7,-0.86603}{14}
        \textab{2.5,-1.06603}{16}
        \textbe{2,-1.53205}{6}
        \end{scope}

        \domath{\newx}{6*3 - \shiftx}
        \domath{\newy}{6*-1.73205 + \shifty}
        \begin{scope}[shift={(\newx,\newy)}]
        \textbe{0,0.2}{12}
        \textab{0.5,0.66603}{15}
        \textbe{1,0.2}{0}
        \textab{1.5,0.66603}{7}
        \textri{1.8,0}{13}
        \textle{1.7,-0.86603}{16}
        \textab{2.5,-1.06603}{1}
        \textbe{2,-1.53205}{8}
        \end{scope}

        \domath{\newx}{7*3 - \shiftx}
        \domath{\newy}{7*-1.73205 + \shifty}
        \begin{scope}[shift={(\newx,\newy)}]
        \textbe{0,0.2}{14}
        \textab{0.5,0.66603}{0}
        \textbe{1,0.2}{2}
        \textab{1.5,0.66603}{9}
        \textri{1.8,0}{15}
        \textle{1.7,-0.86603}{1}
        \textab{2.5,-1.06603}{3}
        \textbe{2,-1.53205}{10}
        \textab{3.5,-1.06603}{2}
        \textbe{3,-1.53205}{16}
        \textab{4.5,-1.06603}{11}
        \textbe{4,-1.53205}{4}
        \textbe{5,-1.53205}{0}
        \textbe{4.5, -2.4}{3}
        \draw (3,-1.73205) -- (5,-1.73205) -- (4.5,-0.86603) -- (3.5,-0.86603) -- (4,-1.73205) -- (4.5,-0.86603)  (5,-1.73205) -- (4.5, -2.598) -- (4,-1.73205);
        \end{scope}
        
        \def\textcol{red}
        \node[red, rotate=60] () at (0.15,0.5) {\scalebox{.9}{
\tiny{$+3$}}};
        \node[red, rotate=300] () at (0.8,0.5) {\scalebox{.9}{
\tiny{$+2$}}};
        \node[red, rotate=60] () at (1.15,0.5) {\scalebox{.9}{
\tiny{$+7$}}};
        \node[red, rotate=300] () at (1.8,0.5) {\scalebox{.9}{
\tiny{$+6$}}};
        \node[red, rotate=240] () at (1.9,-0.4) {\scalebox{.9}{
\tiny{$+3$}}};
        \node[red, rotate=0] () at (2,-0.76603) {\scalebox{.9}{
\tiny{$+2$}}};
\node[red, rotate=240] () at (2.4,-1.2) {\scalebox{.9}{
\tiny{$+7$}}};
 \node[red, rotate=0] () at (2.5,-1.63205) {\scalebox{.9}{
\tiny{$+6$}}};
        \node[blue] () at (.5,0.1) {\tiny{$+5$}};
        \node[blue] () at (1.5,0.1) {\tiny{$+13$}};
        \node[blue] () at (1,0.96) {\tiny{$+9$}};
    \node[blue, rotate=300] () at (1.3,-0.3) {\scalebox{.9}{
\tiny{$+16$}}};

    \end{tikzpicture}}
    \caption{Circular simplicial complex for $n=17$ corresponding to the generating sequence $3\upper52\upper 97\upper{13}\hat 6\upper{16}$}
    \label{fig:n=17}
\end{figure}

Next, we will identify the properties of the sequence that make the previous two examples work in general.

\begin{definition}\label{def: generating sequence}
    A sequence $(a_0,\dots,a_{m-1})$ with elements in $\Z/n\Z$ together with a set $I \subs \lbrace0,\dots,m-1\rbrace$ is a \emph{generating sequence} for $n=4k+1$ if the following properties are satisfied: 
    \begin{enumerate}
        \item\label{enum: coprime sum} $\sum_{i=0}^{m-1} a_i$ is coprime to $n$.
        \item\label{enum: no consecutive turns} For all $i \in I$ we have $i+1 \not\in I$.
        \item\label{enum: no residue twice} $| \{ \pm a_0,\dots,\pm a_{m-1},\pm c_0, \dots, \pm c_{m-1} \} | = 4m$, where $c_i\coleq a_i+a_{i+1}$ if $i \not\in I$ and $c_i\coleq a_{i-1}+a_i+a_{i+1}$ if $i \in I$ and indices are understood modulo $m$, e.g.\@ $a_m\coleq a_0$.
    \end{enumerate}
    For any residue $r \in \{ 1,\dots,2k \}$ that is not contained in $\{ \pm a_0,\dots,\pm a_{m-1},\pm c_0, \dots, \pm c_{m-1} \} $ we say that $r$ is \emph{missing from the generating sequence}. Otherwise, we say that the residue is \emph{covered} by the sequence. We write generating sequences as $a_0 \upper{c_0} a_1 \upper{c_1} \dots a_{m-1} \upper{c_{m-1}}$ where we put $\hat{~}$ over $a_i$ if $i\in I$. In that context, we refer to $(a_1,\dots,a_{m-1})$ as \emph{black sequence} and $(c_1,\dots,c_{m-1})$ as \emph{blue sequence}.
\end{definition}
The examples $1\upper 32\upper 64\upper 5$ and $3\upper 52\upper 97\upper{13}\hat{6}\upper{16}$ are generating sequences with no missing residues for $n=13$ and $n=17$ respectively.

We solely give this definition for $n \equiv 1 \pmod 4$, because only then it is possible that no residue $r$ is missing, which can happen if and only if $m=k$ (cf.\@ property \ref{enum: no residue twice}).

A missing residue $r$ corresponds to a $2$-factor in the complete graph $\binom{\Z/n\Z}{2}$ since $r\ne -r$. This has $n$ edges and is a Hamilton cycle if and only if $r$ is coprime to $n$.

Generating sequences are useful because they give us circular simplicial complexes that contain all edges whose residue is covered by the generating sequence:

\begin{definition}
If $(a_0,\dots, a_{m-1})$ with $I \subs \lbrace0,\dots,m-1\rbrace$ is a generating sequence for $n$, then we define the \emph{corresponding} circular simplicial complex with \labels{} sequence $[x_0,x_1,\dots,x_{mn+1}]$ with $x_i\in \Z/n\Z$ and \layout{} sequence $[y_3,\dots,y_{mn+1}]$ as follows:
We start with $x_0 = 0$, and define recursively \[x_{i+1} = x_i+a_{i\text{  mod $m$}}.\]
For the \layout, we simply define for each $i=3,\dots,mn+1$ that $y_i = 1$ if $(i-1)\text{ mod } m \in I$  and $y_i=0$ otherwise.
\end{definition}
Note that this implies $x_{mn}=n\cdot \sum_{i=0}^{m}a_i=0=x_0$ and $x_{mn+1}=x_1$ and, thus, this will be a circular simplicial complex.

\begin{proposition}\label{prop: with turn}
    Let $k\ge 3$, $n=4k+1$ and suppose that a generating sequence $(a_0,\dots,a_{m-1})$ with $I\subs \lbrace0,\dots,m-1\rbrace$ exists for $n$.
    Then the $(\labels,\layout)$ pair of the corresponding circular simplicial complex is good and it covers exactly those edges whose residue is covered by the generating sequence.
\end{proposition}

The circular simplicial complex obtained with this proposition covers exactly $2mn$ edges.
Therefore, removing any edge that is only in one set of size three and this set creates a simplicial complex that is no longer circular and covers $2mn-1$ edges.
By Observation~\ref{obs:sequences} this simplicial complex has diameter $mn-2$, which is $kn-2 = \tfrac 12 \binom{n}{2}-2$ for $m=k$.
Therefore, the existence of a generating sequence will already be sufficient to prove Theorem~\ref{thm:main, 4k+1} for $n \equiv 1 \pmod 4$.
We note that from a generating sequence we could also directly construct the simplicial complex, but we think that this intermediate step simplifies the arguments and is also helpful for the reader in the remainder of this section to understand the constructions.

\begin{proof}[Proof of Proposition~\ref{prop: with turn}.]
By construction, the circular simplicial complex can only contain edges whose residue appears in the generating sequence.
Let $\lbrace u,v \rbrace\in \binom{\Z/n\Z}{2}$ be an edge whose residue is covered by the generating sequence.
By \ref{enum: no residue twice} of Definition \ref{def: generating sequence}, there is a unique $\ell\in\lbrace0,\dots,m-1\rbrace$ such that $u-v$ or $v-u$ is $a_\ell$, $a_\ell+a_{\ell+1}$ and $\ell\not\in I$, or $a_{\ell-1}+a_\ell+a_{\ell+1}$ and $\ell\in I$.
W.l.o.g.\@ we can assume that it is $u-v$. 

Let us first consider the case that $u-v=a_\ell$.
The simplicial complex contains the edges \[\lbrace x_\ell,x_{\ell+1}\rbrace, \lbrace x_{m+\ell},x_{m+\ell+1}\rbrace,\dots,\lbrace x_{(n-1)m+\ell},x_{(n-1)m+\ell+1}\rbrace.\]
By definition of $x_i$, each of these edges has the same residue as the edge $\lbrace u,v\rbrace$.
Since there are just $n$ edges with this residue, it is enough to show that these edges are pairwise distinct.
Then we know that $\lbrace u,v\rbrace$ must be among these edges. 

Suppose there are distinct $s_1,s_2\in\lbrace 0,\dots,n-1\rbrace$ such that $\lbrace x_{s_1m+\ell},x_{s_1m+\ell+1}\rbrace=\lbrace x_{s_2m+\ell},x_{s_2m+\ell+1}\rbrace$.
\towriteornottowrite{First, we will show that $x_{s_1m+\ell}=x_{s_2m+\ell}$ and $x_{s_1m+\ell+1}=x_{s_2m+\ell+1}$. Assume this is not the case, i.e.\@ $x_{s_1m+\ell}=x_{s_2m+\ell+1}$ and $x_{s_1m+\ell+1}=x_{s_2m+\ell}$. Plugging this into the equation $x_{s_1m+\ell+1} - x_{s_1m+\ell}=a_\ell = x_{s_2m+\ell+1} - x_{s_2m+\ell}$, we get $x_{s_2m+\ell} - x_{s_1m+\ell} = x_{s_1m+\ell} - x_{s_2m+\ell}$ which is a contradiction since $n$ is odd and $a_\ell\ne 0$. Hence, $x_{s_1m+\ell}=x_{s_2m+\ell}$ and $x_{s_1m+\ell+1}=x_{s_2m+\ell+1}$.

Thus, we can conclude}{
Since $x_{s_1m+\ell+1} - x_{s_1m+\ell}=a_\ell = x_{s_2m+\ell+1} - x_{s_2m+\ell}$ and $n$ is odd, we must have $x_{s_1m+\ell}=x_{s_2m+\ell}$ and $x_{s_1m+\ell+1}=x_{s_2m+\ell+1}$.
But by definition, we have}
\[0=x_{s_1m+\ell}-x_{s_2m+\ell}=(s_1-s_2)\cdot\sum_{i=0}^{m-1}a_i\]
and $\sum_{i=0}^{m-1}a_i$ is coprime to $n$ by \ref{enum: coprime sum} of Definition \ref{def: generating sequence}.
Thus, $s_1-s_2$ must be divisible by $n$ which is a contradiction to our assumption on $s_1$ and $s_2$.
Hence, $\lbrace u,v\rbrace$ must appear in the simplicial complex.

The case that $u-v=a_{\ell}+a_{\ell+1}$ and $\ell\not\in I$ is quite similar as the simplicial complex contains the edges \[\lbrace x_\ell,x_{\ell+2}\rbrace, \lbrace x_{m+\ell},x_{m+\ell+2}\rbrace,\dots,\lbrace x_{(n-1)m+\ell},x_{(n-1)m+\ell+2}\rbrace\]
which are all distinct and have endpoints which differ by $a_{\ell}+a_{\ell+1}$.
Thus, $\lbrace u,v\rbrace$ must be among them.

\towriteornottowrite{The final case $u-v=a_{\ell-1}+a_{\ell}+a_{\ell+1}$ and $\ell\in I$ proceeds also quite similar. There, we note that}{Finally, in the case $u-v=a_{\ell-1}+a_{\ell}+a_{\ell+1}$ and $\ell\in I$ we note that} as $\ell-1\not\in I$ the simplicial complex contains the edges
\[\lbrace x_{\ell-1},x_{\ell+2}\rbrace, \lbrace x_{m+\ell-1},x_{m+\ell+2}\rbrace,\dots,\lbrace x_{(n-1)m+\ell-1},x_{(n-1)m+\ell+2}\rbrace\]
which are all distinct and have endpoints which differ by $a_{\ell}+a_{\ell+1}$.
Thus, $\lbrace u,v\rbrace$ must be among them.
\end{proof}

\subsection{The Construction for $n=4k+1$}\label{sec:4k+1}
In this section, we focus on the case $n=4k+1$, which means that we can apply the tools from the previous section directly.
We will construct a simplicial complex with diameter $\frac{1}{2}\binom{n}{2}-2$, which is best possible.

By Proposition~\ref{prop: with turn}, we only have to find a generating sequence (maybe with turns) such that no residue is missing.
This gives a circular simplicial complex that covers all edges of $\binom{\Z/n\Z}{2}$.
By removing an arbitrary edge that is only contained in one set of size three and this set of size three, we get a simplicial complex with diameter $\frac{1}{2}\binom{n}{2}-2$ by Observation~\ref{obs:sequences}.

\begin{lemma} \label{lem:4k+1}
    For $k\geq 3$ there is a generating sequence for $n=4k+1$ of length $k$.
\end{lemma}
\begin{proof}
For $k\leq 5$, the following sequences can be easily checked to fulfill all three conditions of Definition~\ref{def: generating sequence}. 

 For $k=3$ and $n=13$, the sequence $1 \upper{3} 2 \upper{6} 4 \upper {5}$ works.

 For $k=4$ and $n=17$, the sequence $1 \upper{3} 2 \upper{8} 6 \upper {10} 4 \upper{5}$ works.

 For $k=5$ and $n=21$, the sequence $1 \upper{3} 2 \upper{9} 7 \upper {13} 6 \upper{10} 4 \upper{5}$ works.

For $k\ge 6$, we use a different type of sequence depending on the parity of $k$.

{\em Case 1.} Let $k=2\ell+1$ be odd and $\ell\ge 3$. 
We claim that the sequence 
\[(-4\ell+6,4\ell+2,-4,-9,2,1,11,6,15,10,19,14,\dots,4\ell-10,4\ell -1) \, \]
is a generating sequence of length $k$ for $n=4k+1=8\ell+5$. The first seven elements of this sequence are created to work for $n=29=4\cdot (2\cdot 3 +1) +1$. 
Indeed, including the values of the corresponding blue sequence
\[-6 \upper{8} 14 \upper{10} -4 \upper{-13} -9 \upper{-7} 2 \upper{3} 1 \upper{12} 11 \upper{5},\]
it is not difficult to check by hand that all three conditions of Definition \ref{def: generating sequence} are fulfilled. 

This sequence is then generalized for $\ell \geq 4$ and extended, such that starting with $11$ the subsequent elements are always created by alternately subtracting $5$ and adding $9$.
Since there is no $~\hat{~}~$ in the sequence, the corresponding blue sequence is created by summing the neighboring black sequence elements: 
\[(-4\ell+6) \upper{8} (4\ell+2) \upper{4\ell-2} -4 \upper{-13} -9 \upper{-7} 2 \upper {3} 1\upper{12}11\upper{17}6\upper{21}15\upper{25}10\upper{29}19\upper{33}14,\dots,(4\ell-10)\upper{8\ell-11}(4\ell -1)\upper{5}.\]

We have to now check for each $r\in\lbrace 1,\dots, 4\ell+2\rbrace$ that either $r$ or $-r$ appears among the blue or black numbers. One can check this by hand for $r=1,\dots,15$. For the rest, we proceed by the residue modulo $4$. Observe that the black sequence starting at $11$ is periodic modulo $4$. It covers the numbers $11,15,19,\dots, 4\ell-1$, i.e. all which are $3$ (mod 4), as well as the numbers $6,10,14,\dots, 4\ell-10$. These are almost all numbers which are $2$ (mod 4). The remaining three numbers of this residue class, i.e. $-(4\ell-6)$, $4\ell-2$, and $4\ell+2$, appear at the beginning of the sequence.

The blue sequence first covers all numbers which are $1 \pmod 4$ ($17,21,25,\dots, 4\ell +1$), before it continues with $4\ell+5=-(4\ell), 4\ell+9=-(4\ell-4),\dots, 8\ell-11=-16$. Thus, the sequence also covers all numbers which are divisible by 4 and hence fulfills \ref{enum: no residue twice} of Definition \ref{def: generating sequence}. Since $I$ is empty, \ref{enum: no consecutive turns} is fulfilled trivially. For \ref{enum: coprime sum} we have to check that the sum of the elements of the sequence is coprime to $n=8\ell+5$.
The sum is
\begin{align*}(-4\ell+6)+ (4\ell+2)-4&-9+2+1+11+\sum_{i=1}^{\ell-3}(4i+2)+\sum_{i=1}^{\ell-3}(4i+11)\\&=9+2(\ell-3)+4\cdot \frac{(\ell-3)(\ell-2)}{2}+11(\ell-3)+4\cdot \frac{(\ell-3)(\ell-2)}{2}\\
&=9+13(\ell-3)+4\ell^2-20\ell+24\\
&=4\ell^2-7\ell -6.\end{align*}
For the gcd, we get
\begin{align*} 
    \left(4\ell^2-7\ell-6,8\ell+5\right)&=\left(8\ell^2-14\ell-12,8\ell+5\right)
    =(-19\ell-12,8\ell+5)
    =(-3\ell-2,8\ell+5)\\
    &=(-3\ell-2,-\ell-1)
    =1.
\end{align*}
This finishes the proof when $n$ is of the form $8\ell +5$.

{\em Case 2.} Let $k=2\ell$ be even and $\ell \ge  3$. 
Here we were not able to find a generating sequence of length $k$ without turns for all even $k$.
But if we allow just a single turn, then there is such a sequence:
\[(3,2,7,6,\dots, 4\ell-1,\widehat{4\ell-2}).\]
This sequence is a generalization of our example sequence \towriteornottowrite{in Figure~\ref{fig:n=17}}{above} for $n=17$ such that starting with $3$ the subsequent elements are always created by alternately subtracting $1$ and adding $5$. The lone turn happens at the last element. It turns out that this is a generating sequence of length $k$ for $n=8\ell+1$. Including the blue elements we have 
    \[3\upper 52\upper 97\upper{13}6\dots4\ell-6\upper{8\ell-7}4\ell-1\upper{8\ell-3}\hat{4\ell-2}\upper{8\ell}.\]
 To verify \ref{enum: no residue twice} of Definition \ref{def: generating sequence} we again consider numbers based on residue modulo $4$. 
Observe that the black sequence is periodic modulo $4$. It covers the numbers $3,7,11,\dots, 4\ell-1$, i.e. all which are $3$ (mod 4), as well as the numbers $2,6,10,\dots, 4\ell-2$, i.e. all which are $2 \pmod 4$.
  The blue sequence first covers numbers at least $5$ with residue $1 \pmod 4$ ($5,9,\dots, 4\ell-3$), before it continues with $4\ell+1=-(4\ell),4\ell+5=-(4\ell-4),\dots,8\ell-3=-4$, i.e.\@ it covers all numbers with residue $0 \pmod 4$. The only residue missing is $1 = -8\ell$. This is covered by the turn at the very end.

    Since $\abs I = 1$, \ref{enum: no consecutive turns} of the definition is satisfied trivially. For \ref{enum: coprime sum}, we again check that the sum of the $a_i$ is coprime to $n$. The sum is
    \[\sum_{i=1}^\ell(4i-1)+\sum_{i=1}^\ell(4i-2)=4\cdot \ell(\ell+1)-3\ell=4\ell^2+\ell.\]
    For the gcd, we get
    \begin{equation*}
        \left(4\ell^2+\ell,8\ell+1\right)=\left(8\ell^2+2\ell,8\ell+1\right)
        =(\ell,8\ell+1)
        =1.
    \end{equation*}
    This finishes the proof.
\end{proof}
This lemma together with Proposition~\ref{prop: with turn} and Observation~\ref{obs:sequences} proves the case $n=4k+1$ of Theorem~\ref{thm:main, 4k+1} when $k \ge 3$.

\subsection{The Construction for $n=4k+4$}

It is not possible to find a generating sequence with no missing residues when $n \equiv 0 \pmod 4$, because for the former we need the number of non-zero elements in $\Z/n\Z$ to be divisible by 4.
Instead, we will give an optimal construction for the case $n=4k+4$ by extending a generating sequence for $n'=4k+1$. 
Note that when $n \equiv 0 \pmod 4$, $\binom n2$ is even, and, hence, we again have one edge to spare.

We let the vertex set be $(\Z/n'\Z) \cup \{ a,b,c \}$ where $a,b,c\not\in\Z/n'\Z$ are three distinct symbols.
By the previous section, we can find a generating sequence for $n'$ and get a circular simplicial complex that covers all edges of $\binom{\Z/n'\Z}{2}$.
But then the edges left uncovered form a $K_{n',3}$ and a triangle on the smaller side \towriteornottowrite{(i.e.\@ the part of the bipartition of size 3)}{}. This would be a linear loss compared to the upper bound and there would be no ``space'' to extend the simplicial complex significantly since every new triangle must intersect $\lbrace a,b,c\rbrace$ in two vertices.
Therefore, we will instead find a generating sequence that misses two residues $r_1$ and $r_2$, that are coprime to $n'$, and obtain a circular simplicial complex that covers all edges of $\binom{\Z/n'\Z}{2}$ except for two Hamilton cycles.
Now the remaining edges form a $K_{n',3}$ together with two cycles of length $n'$ on the larger side \towriteornottowrite{(i.e.\@ the part of the bipartition of size $n'$)}{} and a triangle on the smaller side. As it turns out, this already provides enough ``wiggle room'' to extend the circular simplicial complex to cover all edges but one. 
For this, we need to be able to cut the circular complex in an appropriate way. The generating sequence for which this is possible is provided by the next lemma. 

\begin{lemma}\label{lem: gen sequence, 1 and 2 missing}
    For every $k$ with $k\ge 4$, there exists a generating sequence for $n'=4k+1$ with the following properties:
    \begin{itemize}
        \item Exactly the residues 1 and 2 are missing.
        \item It is possible to cut the resulting circular simplicial complex such that $\lbrace0,4k-2\rbrace$ is at one end of the simplicial complex.
        \item If $k\ge 5$, it is possible to cut the resulting circular simplicial complex at $\lbrace0,13\rbrace$ such that the edge $\lbrace0,7\rbrace$ is at one end of the simplicial complex. 
    \end{itemize}
\end{lemma}
The second condition is what we need in this section.
The third condition will only be used later when we handle the case $n \equiv 3 \pmod 4$.
\begin{proof}
    The following table gives generating sequences for small values of $n'$ obeying all conditions:
    
    \begin{table}[H]
\centering
\begin{tabular}{c|c}
   $n'$  & generating sequence\\\hline
    17& $3\upper74\upper{-5}8\upper{-6}$\\
    21&  $4\upper{-10}7\upper{-8}6\upper{-3}(-9)\upper{-5}$\\
    25& $4\upper{-11}10\upper{-8}7\upper{-12}6\upper{-3}(-9)\upper{-5}$\\
    29& $11\upper{-10}8\upper{-4} (-12)\upper{-5}7\upper{13}6\upper93\upper{14}$\\
\end{tabular}
\end{table}
    To verify the second and third conditions of this lemma, note the following: The corresponding circular simplicial complexes of the generating sequences do not have any turns. Hence, they look like the circular simplicial complex in Figure~\ref{fig:n=13}. There, we can see that we can cut the circular simplicial complex at any edge whose residue is in the blue sequence. If we cut the circular simplicial complex at a blue edge, the neighboring black and blue edges form the end of the simplicial complex. Hence, we can always cut in a way that $\lbrace 0,4k-2\rbrace$ is at one end of the simplicial complex. 

    The third condition of the lemma is satisfied if residue 7 is followed by residue 6 in the black sequence. This is the case for sequences above except for $n'=17$. Hence, the third condition is only valid for $k\ge 5$.
    
    For larger $k$, we again distinguish two cases based on the parity of $k$.

    \textit{Case 1.} Let $k=2\ell+1$ be\towriteornottowrite{}{t} odd and $\ell\ge 4$. We will show that \[8\upper{3}-5\upper{4\ell-6}\hat{(4\ell-1)}\upper{4\ell-2}4\upper{-12}-16\upper{-9}7\upper{13}6\upper{17}11\upper{21}10\upper{25}15\upper{29}14\dots(4\ell-10)\upper{-20}(4\ell-5)\upper{4\ell+3}\] 
    is a generating sequence for $n'=8\ell+5$ and fulfills the three conditions of the lemma. The sequence starts with an ad-hoc part but becomes periodic modulo 4 from the element 7 onwards. Starting with 7, the subsequent elements of the black sequence are always created by alternately subtracting $1$ and adding 5.
    
    As before, we check that all residues except for 1 and 2 are covered exactly once: The black sequence covers almost all numbers which are $3$ mod $4$ via $7,11,\dots,4\ell-5$. The residues 3 and $4\ell-1$ appear at the beginning of the sequence. Furthermore, the black sequence covers almost every number which is $2$ mod $4$ via $6,10,\dots,4\ell-10$. The residues $4\ell-6$ and $4\ell-2$ appear at the beginning and $4\ell+2=-(4\ell+3)$ at the end of the sequence. The blue numbers starting with $13,17,\dots,4\ell+1$ cover almost all residues that are $1$ mod 4 and finish with $4\ell+5=-(4\ell), 4\ell+9=-(4\ell-4),\dots,8\ell -15=-20$ covering almost  all residues that are $0$ mod 4. The residues $5$, $9$, $16$, 12, 8, and 4 can be found at the beginning of the sequence. This shows \ref{enum: no residue twice} of Definition \ref{def: generating sequence}. \ref{enum: no consecutive turns} is satisfied because $\abs I = 1$. To check \ref{enum: coprime sum}, i.e.\@ that the sum of the $a_i$ is coprime to $n'=8\ell+5$ we compute the sum
    \begin{align*}8-5+(4\ell-1)+4-16+7+\sum_{i=1}^{\ell-3}(4i+2) + \sum_{i=1}^{\ell-3}(4i+7)&=-3+4\ell+9\ell-27+4\ell^2-20\ell +24\\
    &=4\ell^2-7\ell-6
    \end{align*}
    and \begin{equation*}(4\ell^2-7\ell-6,8\ell+5)=(8\ell^2-14\ell-12,8\ell+5)=(-19\ell-12,8\ell+5)=(5\ell+3,8\ell+5)=1.
    \end{equation*}

    Hence, the sequence is indeed a generating sequence where exactly the residues 1 and 2 are missing.
    
    To verify the second condition of the lemma, we note that if we cut the corresponding circular simplicial complex at an edge with residue $4\ell+3$, an edge with residue $3=-(4k-2)$ will be at one end of the simplicial complex. Hence, there is an edge with residue $4\ell+3$ such that if we cut at that edge, the edge $\lbrace 0,4k-2\rbrace$ will be at one end of the simplicial complex.

    For the third condition of the lemma, we notice that residue 7 is followed by residue 6 in the black sequence and there is no turn after residue 7 in the corresponding circular simplicial complex. This is enough to guarantee that we can cut at $\lbrace 0,13\rbrace$ such that the edge $\lbrace0,7\rbrace$ is at one end of the simplicial complex. 

    This completes the proof for the case that $k$ is odd.

    \textit{Case 2.} Let $k=2\ell$ be even with $\ell\ge 4$. There, we claim that the sequence \[\widehat 4\upper{8}(4\ell-6)\upper{4\ell-2}9\upper{3}-12\upper{5}7\upper{13}6\upper{17}11\upper{21}10\upper{25}15\upper{29}14~\dots~ (4\ell-10)\upper{-16}(4\ell-5)\upper{4\ell-1}\] works. Note how starting with $7,6,11,10,\dots$, the black sequence is the same as in the $k$ being odd case above again alternately subtracting 1 and adding 5. 
    
    We start by checking that it is indeed a generating sequence for $n'=8\ell+1$ that misses only the residues 1 and 2. 
    
    The black sequence covers almost all numbers which are $3$ mod $4$ via $7,11,\dots,4\ell-5$. The residues 3 and $4\ell-1$ appear at the beginning 
    and the end of the sequence respectively. Furthermore, the black sequence covers almost every number which is $2$ mod $4$ via $6,10,\dots,4\ell-10$. The residues $4\ell-6$ and $4\ell-2$ appear at the beginning of the sequence. The blue numbers starting with $13,17,\dots,4\ell-3$ cover almost all residues that are $1$ mod 4 and finish with $4\ell+1=-(4\ell), 4\ell+5=-(4\ell-4),\dots,8\ell -15=-16$ covering almost  all residues that are $0$ mod 4. The residues $5$, $9$, 12, 8, and 4 can be found at the beginning of the sequence. This shows \ref{enum: no residue twice} of Definition \ref{def: generating sequence}. \ref{enum: no consecutive turns} is satisfied because $\abs I = 1$.
    
    Hence, only \ref{enum: coprime sum} of Definition \ref{def: generating sequence} remains to be checked.
    The sequence has the sum
\begin{align*}
    4+(4\ell-6)+9-12+7+\sum_{i=1}^{\ell-3}(4i+2)+\sum_{i=1}^{\ell-3}(4i+7)&=2+4\ell+9\ell-27+4\ell^2-20\ell+24\\
    &=4\ell^2-7\ell-1
\end{align*}
and
\begin{equation*}(4\ell^2-7\ell-1,8\ell+1)=(8\ell^2-14\ell-2,8\ell+1)=(-15\ell-2,8\ell+1)=(\ell,8\ell+1)=1.
\end{equation*}
Therefore, this is a generating sequence for $n'=8\ell+1$ only missing the residues 1 and 2.

 To verify the second condition of the lemma, we note that if we cut the corresponding circular simplicial complex at an edge with residue $5$, an edge with residue $3=-(4k-2)$ will be at one end of the simplicial complex. Hence, there is an edge with residue $5$ such that if we cut at that edge, the edge $\lbrace 0,4k-2\rbrace$ will be at one end of the simplicial complex.

    For the third condition of the lemma, we notice that again residue 7 is followed by residue 6 in the black sequence and there is no turn after residue 7 in the corresponding circular simplicial complex. 

    This completes the proof.
\end{proof}

After going through the circular simplicial complex provided by the previous lemma, we plan to construct an extension that attaches through the edge $\{ 0,4k-2\}$ and covers all the edges of the remaining graph \towriteornottowrite{ on the vertex set $\Z/n'\Z\cup\{a,b,c\}$}{}.

To cover the remaining edges we make use of two simple constructions.
\towriteornottowrite{\begin{definition}
For distinct vertices $u,v_1,\dots,v_{t}$, we call the simplicial complex containing the triples $u,v_i,v_{i+1}$ for $i=1,\dots,t-1$ a \emph{rotation} of the sequence $v_1,\dots,v_t$ around $u$. 
The \labels{} of this simplicial complex are $[u,v_1,v_2,v_3,\dots,v_t]$ and the \layout{} is all $1$'s.

For distinct vertices $u,w, v_1,\dots,v_{t}$, we call the \emph{zig-zag} of the sequence $v_1,\dots,v_t$ around $(u,w)$ the simplicial complex containing the triples $v_iv_{i+1}u$ for $i \not\equiv 2 \pmod 4$ and $v_iv_{i+1}w$ for $i \not\equiv 0 \pmod 4$ with $i=1,\dots,t-1$.
Here, the \labels{} sequence of the simplicial complex is $[u,v_1,v_2,w,v_3,v_4,u,v_5,v_6,w,v_7,v_8,\dots]$ and the \layout{} has all $0$'s.
\end{definition}
Note that the rotation has diameter $t-2$ and the zig-zag has diameter $\lfloor \frac{3t}2-2 \rfloor$ for $t \ge 2$.}{
    For a path $v_1,\dots,v_{t}$, we call the simplicial complex containing the triples $a,v_i,v_{i+1}$ for $i=1,\dots,t-1$ a \emph{rotation} around $a$. This has diameter $t-2$.
The \labels{} of this simplicial complex are $[a,v_1,v_2,v_3,\dots,v_t]$ and the \layout{} is all $1$'s.
Similarly, for a path $v_1,\dots,v_{t}$ we call the \emph{zig-zag} around $(b,c)$ the simplicial complex containing the triples $v_iv_{i+1}b$ for $i \not\equiv 2 \pmod 4$ and $v_iv_{i+1}c$ for $i \not\equiv 0 \pmod 4$ with $i=1,\dots,t-1$, which has diameter $\lfloor 3t/2-2 \rfloor$ for $t \ge 2$.
Here the \labels{} sequence of the simplicial complex is $[b,v_1,v_2,c,v_3,v_4,b,v_5,v_6,c,v_7,v_8,\dots]$ and the \layout{} has all $0$'s.}

Recall that the remaining graph is the union of a $K_{n',3}$, two edge-disjoint cycles of length $n'$ on the \towriteornottowrite{larger}{bigger} side and a triangle on the smaller side. 
In order to cover most of the remaining edges we choose a long path in each of the two cycles of length $n'$, do a rotation around $a$ using one of them, and a zig-zag around $(b,c)$ using the other, so to cover most edges of the $K_{n',3}$. After patching them up again and connecting them to the circular complex, we obtain a simplicial complex that misses only a small constant number of edges.
To circumvent this small error, the eventual construction will be slightly modified.

We demonstrate the main parts of this approach by describing a simplicial complex on $n=16$ vertices which leaves only nine edges uncovered. 
Figure~\ref{fig:n=16} shows a good simplicial complex on $n'=13$ vertices corresponding to a generating sequence missing residues $5$ and $6$ (the skew part). This complex attaches through the edge $\lbrace0,1\rbrace$ to the rotation with the \towriteornottowrite{sequence (depicted in red)}{red path}
\[ 0,6,12,5,11,4,10,3,9,2,8 \]
around $a$ and then continuing with the zig-zag of the \towriteornottowrite{sequence (depicted in blue)}{blue path}
\[ 0,5,10,2,7,12,4,9,1,6,11,2 \]
around $(b,c)$.
Note that the first path only uses residue $r_1=6$ and the second path only uses residue $r_2=5$.

\begin{figure}[H]
    \centering
    \begin{tikzpicture}[scale=0.8]
    \tikzstyle{every node}=[font=\footnotesize]
        \foreach \x in {0,...,5} {
        \domath{\newx}{\x*1.5}
        \domath{\newy}{-\x*0.86603}
        \begin{scope}[shift={(\newx,\newy)}]
        \draw (0,0) -- (1,0) -- (.5,0.86603) -- (0,0) -- (.5,-0.86603) -- (1,0) -- (1.5,-0.86603) -- (2,0) (.5,-0.86603) -- (1.5,-0.86603) (1,0) -- (2,0);
        \end{scope}
        }
        \textbe{0,0.1}{0}
        \textab{0.5,0.7}{1}
        \textab{1,-0.1}{3}
        \textbe{.5,-0.7}{4}
        \begin{scope}[shift={(1.5,-0.86603)}]
        \textbe{0,0.1}{6}
        \textab{0.5,0.7}{7}
        \textab{1,-0.1}{9}
        \textbe{.5,-0.7}{10}
        \end{scope}
        \begin{scope}[shift={(3,-1.73206)}]
        \textbe{0,0.1}{12}
        \textab{0.5,0.7}{0}
        \textab{1,-0.1}{2}
        \textbe{.5,-0.7}{3}
        \end{scope}
        \begin{scope}[shift={(4.5,-2.598)}]
        \textbe{0,0.1}{5}
        \textab{0.5,0.7}{6}
        \textab{1,-0.1}{8}
        \textbe{.5,-0.7}{9}
        \end{scope}
        \begin{scope}[shift={(6,-3.46412)}]
        \textbe{0,0.1}{11}
        \textab{0.5,0.7}{12}
        \textab{1,-0.1}{1}
        \textbe{.5,-0.7}{2}
        \end{scope}
        \begin{scope}[shift={(7.5,-4.33015)}]
        \textbe{0,0.1}{4}
        \textab{0.5,0.7}{5}
        \textab{1,-0.1}{7}
        \textbe{.5,-0.66}{8}
        \end{scope}
        \begin{scope}[shift={(9,-5.19618)}]
        \textbe{0,0.2}{10}
        \textab{0.5,0.7}{11}
        \textab{1,-0.1}{0}
        \draw (0,0) -- (1,0) -- (0.5,.86603);
        \end{scope}
        \draw (0,0) -- (-.5,0.86603) -- (.5,0.86603);
        \begin{scope}[shift={(-.5,0.86603)}]
        \foreach \x in {0,...,9} {
        \domath{\angl}{-\x*25-60}
        \domath{\nangl}{-(\x+1)*25-60}
        \draw[red, thick] (\angl:1) -- (\nangl:1);
        \draw (\nangl:1)-- (0,0);
        }
        \draw[red, thick] (1,0) -- (.5,-0.86603);
        \def\textcol{black}
        \textat{-85:1.2}{6}
        \textat{-110:1.2}{12}
        \textat{-135:1.2}{5}
        \textat{-160:1.2}{11}
        \textat{-185:1.2}{4}
        \textat{-210:1.2}{10}
        \textat{-235:1.2}{3}
        \textat{-260:1.2}{9}
        \textat{-280:1.2}{2}
        \textat{0.3,-0.15}{$a$}
        \textat{0.7,.65}{8}
        \begin{scope}[shift={(50:1)}]
        \draw (-130:1) -- (20:1) -- (0,0) -- (1.44,-.52603);
        \begin{scope}[shift={(20:1)}]
        \foreach \x in {0,...,7} {
        \domath{\y}{\x+1}
        \domath{\z}{\x+.5}
        \domath{\w}{\x+1.5}
        \draw (\x,0) -- (\y,0) -- (\z,-.86603) -- (\x,0) (\z,-.86603) -- (\w,-.86603);
        }
        \draw (8,0) -- (9,0) -- (8.5,-.86603) -- (8,0);
        \draw[blue, thick] (.5,-.86603) -- (1,0) -- (2,0) -- (2.5,-.86603) -- (3.5,-.86603) -- (4,0) -- (5,0) -- (5.5,-.86603) -- (6.5,-.86603) -- (7,0) -- (8,0) -- (8.5,-.86603);
        \def\textcol{black}
        \textat{0,.2}{$b$}
        \textat{0.5,-1.1}{0}
        \textat{1.5,-1.1}{$c$}
        \textat{2.5,-1.1}{2}
        \textat{3.5,-1.1}{7}
        \textat{4.5,-1.1}{$c$}
        \textat{5.5,-1.1}{9}
        \textat{6.5,-1.1}{1}
        \textat{7.5,-1.1}{$c$}
        \textat{8.5,-1.1}{3}
        \textat{1,.2}{5}
        \textat{2.,.2}{10}
        \textat{3.,.2}{$b$}
        \textat{4.,.2}{12}
        \textat{5,.2}{4}
        \textat{6.,.2}{$b$}
        \textat{7.,.2}{6}
        \textat{8.,.2}{11}
        \textat{9.,.2}{$b$}
        \end{scope}
        \end{scope}
        \end{scope}
    \end{tikzpicture}
    \caption{Simplicial complex for $n=16$. Only nine edges $\lbrace 0,7\rbrace$, $\lbrace 1,7\rbrace$, $\lbrace 1,8\rbrace$, $\lbrace 1,10\rbrace$, $\lbrace 3,8\rbrace$, $\lbrace 7,a\rbrace$, $\lbrace8,c\rbrace$, $\lbrace a,c\rbrace$ and $\lbrace b,c\rbrace$ are missing.}
    \label{fig:n=16}
\end{figure}
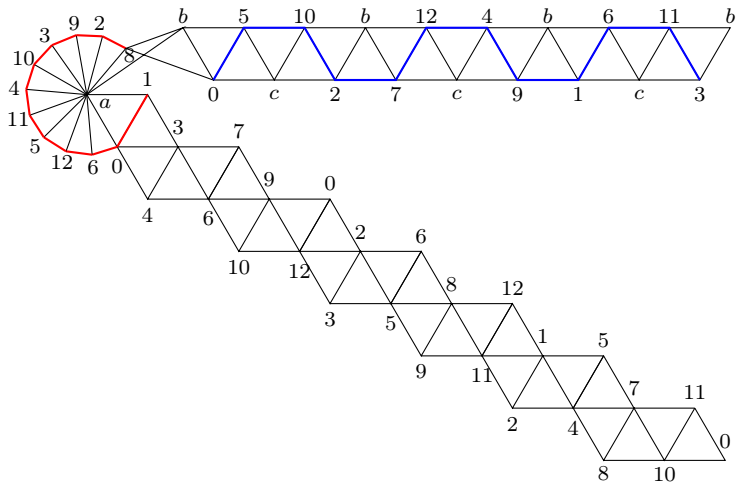

In the proof below, we describe how to adjust this approach to obtain a simplicial complex with no uncovered edges.  

\begin{proof}[Proof of Theorem~\ref{thm:main, 4k+1} for $n=4k+4$ when $k \ge 4$.]
Let $n'=n-3=4k+1$ and consider the vertex set $(\Z/n'\Z) \cup \{ a,b,c \}$ where $a,b,c\not\in\Z/n'\Z$ are three distinct symbols.

By Lemma \ref{lem: gen sequence, 1 and 2 missing}, there is a generating sequence for $n'$ such that only the residues 1 and 2 are missing and whose corresponding circular simplicial complex $\towriteornottowrite{\C}{C}$ on $n'$ vertices can be cut in such a way that $\lbrace 0,4k-2\rbrace$ is at one end of the complex. By this cutting, one edge will be lost. This edge will be the only edge we will not cover in our final complex. 
Our goal is to add triangles such that every edge that corresponds to the residues 1 or 2 and every edge containing $a$, $b$, or $c$ is used exactly once. Then Proposition~\ref{prop: with turn} implies the final simplicial complex covers all edges except for one.

Figure~\ref{fig:4n+4} shows what these triangles look like. Note that the leftmost triangle of the figure has the edge $\lbrace 0,4k-2\rbrace$ and belongs to $\towriteornottowrite{\C}{C}$. Since we cut $\towriteornottowrite{\C}{C}$ in such a way that this edge appears at one end, we can attach our new triangles there. 

\begin{figure}[H]
    \centering
    \resizebox{\linewidth}{!}{
    \begin{tikzpicture}
    \tikzstyle{every node}=[font=\footnotesize]
        \draw[dashed] (.5,0.86603) -- (0,0) -- (1,0);
        \draw (.5,0.86603) -- (1,0) (.5,0.86603) -- (1.5,0.86603) (1,0) -- (2,0) -- (1.5,0.86603);
        \draw[red, thick] (1.5,0.86603) -- (1,0) (2.5,0.86603) -- (3,0);
        \draw [red, thick, domain=180:360] plot ({2+cos(\x)}, {sin(\x)}); 
        \draw (2,0) -- ({2+cos(360)},{sin(360});
        \draw (2,0) -- (2.5,0.86603) (3,0) -- (4,0) -- (2.5,0.86603) -- (4.5,0.86603) -- (4,0);
        \foreach \x in {0,...,1} {
            \draw ({4.5+\x}, 0.86603) -- ({5+\x}, 0) -- ({4+\x}, 0) ({5+\x}, 0) -- ({5.5+\x}, 0.86603)-- ({4.5+\x}, 0.86603);
        }
        \textat{0.5,1.06603}{$4k-2$}
        \textat{1.5,1.26603}{$4k-1$}
        \textat{.9,-.2}{0}
        \textat{2,.2}{$a$}
        {\def\textcol{red}
        \textat{2,-1.3}{$\substack{1,2,\dots,4k-8,4k-6,4k-7,\\4k-5,4k-3,4k-4,c,4k-2}$}
        }
        \textat{3.1,-.2}{$4k$}
        \textat{2.5,1.06603}{$b$}
        \textat{4.5,1.06603}{2}
        \textat{5.5,1.06603}{$4$}
        \textat{6.5,1.06603}{$b$}
        \textat{8.5,1.06603}{$7$}
        \textat{9.5,1.06603}{$b$}
        \textat{10.5,1.06603}{$1$}
        \textat{11.5,1.06603}{$4k$}
        \textat{4,-.2}{0}
        \textat{5,-.2}{$c$}
        \textat{6,-.2}{6}
        \draw(6,0) -- (6.5,0);
        \textat{9,-.2}{5}
        \textat{10,-.2}{3}
        \textat{11,-.2}{$c$}
        \textat{12.5,0}{$4k-1$}
        \textat{11.,-0.86603}{$4k-3$}
        \textat{12.5,-0.66603}{$b$}
        \textat{13.5,-0.66603}{$4k-6$}
        \textat{14.5,-0.66603}{$c$}
        \textat{12,-1.86603}{$4k-2$}
        \textat{13,-1.86603}{$4k-4$}
        \textat{14,-1.86603}{$4k-5$}
        \textat{15,-1.86603}{$b$}

        \begin{scope}[shift={(4,0)}]
        \foreach \x in {0,...,2} {
            \draw ({4.5+\x}, 0.86603) -- ({5+\x}, 0) ({5+\x}, 0) -- ({5.5+\x}, 0.86603)-- ({4.5+\x}, 0.86603);
        }
        \draw[dotted] (2,0) -- (4.5,0) (2,.86603) -- (4.5,.86603); 
        \draw (4.5,0) -- (8,0);
        \draw (7,0) -- (8,0) -- (7.5, 0.86603) (7,0) -- (7.5, -0.86603) -- (8,0) -- (8.5,-.86603) -- (7.5, -0.86603) -- (8, -1.73206) -- (8.5,-.86603) -- (9, -1.73206) -- (8, -1.73206) (8.5,-.86603) -- (9.5,-.86603) -- (9, -1.73206) -- (10, -1.73206) -- (9.5,-.86603) -- (10.5,-.86603)-- (10, -1.73206)-- (11, -1.73206)--(10.5,-.86603);
        \draw[blue, thick] (0,0) -- (.5, 0.86603) -- (1.5, 0.86603) -- (2,0) -- (2.5,0) (4.5, 0.86603) -- (5,0) -- (6,0) -- (6.5, 0.86603);
        \end{scope}
        {\def\textcol{blue}
        \textat{7.5,-1}{$\substack{8,10,12,\dots,4k-10,4k-8,\\4k-7,4k-9,\dots,11,9}$}}
        \draw[->] (7.5,-.5) -- (7.5,.5);
    \end{tikzpicture}}
    \caption{The construction for $n=4k+4$}
    \label{fig:4n+4}
\end{figure}
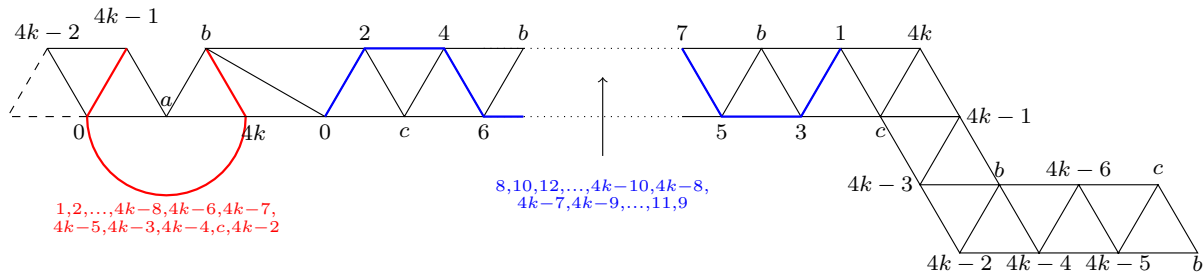

The red circle segment indicates the rotation of the \towriteornottowrite{sequence}{path}
\[4k-1,0,1,2,3,4,\dots,4k-8,4k-6,4k-7,4k-5,4k-3,4k-4,c,4k-2,4k,b\]
around $a$ where the dots indicate that we always increase by 1 until we reach $4k-8$. The rotation is attached to $\towriteornottowrite{\C}{C}$ with the triple $4k-2,0,4k-1$.
Note that most of this is contained in the Hamilton cycle given by the residue $1$, but that there are some changes and, in particular, $b$ and $c$ are also there.

After this rotation around $a$, we want to continue with the zig-zag around $(b,c)$.
For this, we will use the \towriteornottowrite{sequence}{path}
\[ 0,2,4,6,\dots, 4k-8, 4k-7, 4k-9, \dots,1 \]
where the first dots indicate that we always increase by 2 until we reach $4k-8$ and the second dots indicate that we always decrease by 2 until we reach 1. The zig-zag is connected to the rest using the triple $b,4k,0$.
The remaining edges can be covered using the constant-sized ad-hoc construction continuing after the triple $3,1,c$, see Figure~\ref{fig:4n+4}.

It remains to check that all edges with residue $1$ or $2$ as well as 
all edges containing $a$, $b$, or $c$ are present exactly once. There are $2n'$ edges with residue 1 and 2 and $3n'+3$ edges containing $a$, $b$ and $c$. Hence, in total, there are $5n'+3=20k+8$ edges.  The check can be done by hand by first checking that the number of edges added by the new triangles is indeed $20k+8$ and then checking that each such edge appears in the figure at least once. We do this in full detail in Appendix~\ref{sec: Check for 4n+4}.

We can conclude that the simplicial complex is good and covers $\binom{n}{2}-1$ edges. By Observation~\ref{obs:sequences}, the diameter of the simplicial complex is $\frac{1}{2}\binom{n}{2}-2$.
\end{proof}

\subsection{The Construction for $n=4k+3$}
If $n=4k+3$, the situation is a bit more complicated, but we will proceed similarly as in the previous section.
Now, $\binom{n}{2}$ is odd\towriteornottowrite{,}{} which means that in an optimal simplicial complex, each edge must be used exactly once and we do not have an edge to spare. 

As shown in Figure~\ref{fig:n=15} for $n=15$ the idea is to let $n'=n-2=4k+1$ and to consider the vertex set $\Z/n'\Z \cup \{ a,b\}$ where $a,b\not\in\Z/n'\Z$ are distinct symbols. Now, we find a generating sequence for $n'$ missing two residues $r_1$ and $r_2$ that are coprime to $n'$.
This defines a circular simplicial complex, which is then cut and we use a rotation around $a$ and $b$ with \towriteornottowrite{sequences}{paths} given by $r_1$ and $r_2$.
This easily gives a simplicial complex with diameter a small constant below the optimum.

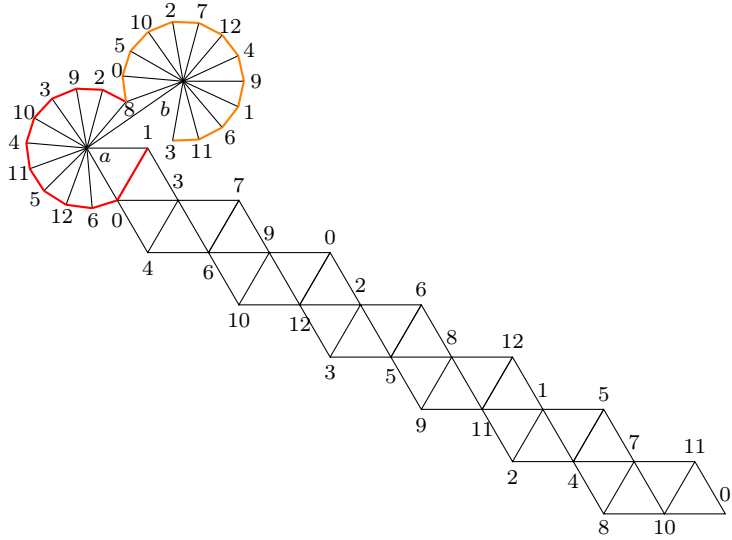
\begin{figure}[H]
    \centering
    \begin{tikzpicture}[scale=0.8]
    \tikzstyle{every node}=[font=\footnotesize]
        \foreach \x in {0,...,5} {
        \domath{\newx}{\x*1.5}
        \domath{\newy}{-\x*0.86603}
        \begin{scope}[shift={(\newx,\newy)}]
        \draw (0,0) -- (1,0) -- (.5,0.86603) -- (0,0) -- (.5,-0.86603) -- (1,0) -- (1.5,-0.86603) -- (2,0) (.5,-0.86603) -- (1.5,-0.86603) (1,0) -- (2,0);
        \end{scope}
        }
        \textbe{0,0.1}{0}
        \textab{0.5,0.7}{1}
        \textab{1,-0.1}{3}
        \textbe{.5,-0.7}{4}
        \begin{scope}[shift={(1.5,-0.86603)}]
        \textbe{0,0.1}{6}
        \textab{0.5,0.7}{7}
        \textab{1,-0.1}{9}
        \textbe{.5,-0.7}{10}
        \end{scope}
        \begin{scope}[shift={(3,-1.73206)}]
        \textbe{0,0.1}{12}
        \textab{0.5,0.7}{0}
        \textab{1,-0.1}{2}
        \textbe{.5,-0.7}{3}
        \end{scope}
        \begin{scope}[shift={(4.5,-2.598)}]
        \textbe{0,0.1}{5}
        \textab{0.5,0.7}{6}
        \textab{1,-0.1}{8}
        \textbe{.5,-0.7}{9}
        \end{scope}
        \begin{scope}[shift={(6,-3.46412)}]
        \textbe{0,0.1}{11}
        \textab{0.5,0.7}{12}
        \textab{1,-0.1}{1}
        \textbe{.5,-0.7}{2}
        \end{scope}
        \begin{scope}[shift={(7.5,-4.33015)}]
        \textbe{0,0.1}{4}
        \textab{0.5,0.7}{5}
        \textab{1,-0.1}{7}
        \textbe{.5,-0.66}{8}
        \end{scope}
        \begin{scope}[shift={(9,-5.19618)}]
        \textbe{0,0.2}{10}
        \textab{0.5,0.7}{11}
        \textab{1,-0.1}{0}
        \draw (0,0) -- (1,0) -- (0.5,.86603);
        \end{scope}
        \draw (0,0) -- (-.5,0.86603) -- (.5,0.86603);
        \draw[thick, red] (0,0) -- (.5,0.86603);
        \begin{scope}[shift={(-.5,0.86603)}]
        \foreach \x in {0,...,9} {
        \domath{\angl}{-\x*25-60}
        \domath{\nangl}{-(\x+1)*25-60}
        \draw[red, thick] (\angl:1) -- (\nangl:1);
        \draw (\nangl:1) -- (0,0);
        }
        \def\textcol{black}
        \textat{-85:1.2}{6}
        \textat{-110:1.2}{12}
        \textat{-135:1.2}{5}
        \textat{-160:1.2}{11}
        \textat{-185:1.2}{4}
        \textat{-210:1.2}{10}
        \textat{-235:1.2}{3}
        \textat{-260:1.2}{9}
        \textat{-280:1.2}{2}
        \textat{0.3,-0.15}{$a$}
        \textat{0.7,.65}{8}
        \begin{scope}[shift={(50:1)}]
        \draw (-130:1) -- (20:1) -- (0,0);
        \begin{scope}[shift={(20:1)}]
        \foreach \x in {0,...,11} {
        \domath{\angl}{-\x*25+200}
        \domath{\nangl}{-(\x+1)*25+200}
        \draw[thick, orange] (\angl:1) -- (\nangl:1);
        \draw (\nangl:1) -- (0,0);
        }
        \def\textcol{black}
        \textat{-190:1.1}{0}
        \textat{-210:1.2}{5}
        \textat{-235:1.2}{10}
        \textat{-260:1.2}{2}
        \textat{-285:1.2}{7}
        \textat{-310:1.2}{12}
        \textat{-335:1.2}{4}
        \textat{-360:1.2}{9}
        \textat{-385:1.2}{1}
        \textat{-410:1.2}{6}
        \textat{-435:1.2}{11}
        \textat{-460:1.2}{3}
        \textat{-0.3,-0.45}{$b$}
        \end{scope}
        \end{scope}
        \end{scope}
    \end{tikzpicture}
    \caption{Simplicial complex for $n=15$. Six edges $\lbrace 0,7\rbrace,\lbrace 1,7\rbrace,\lbrace 1,8\rbrace,\lbrace 1,10\rbrace,\lbrace 3,8\rbrace$ and $\lbrace 7,a\rbrace$ are missing.}
    \label{fig:n=15}
\end{figure}

Since in the end, we want to use up every edge, we also have to use the edge that gets destroyed in the cutting of the simplicial complex somewhere else. Therefore, we have to be more careful where we insert the cut.

\begin{proof}[Proof of Theorem~\ref{thm:main, 4k+1} for $n=4k+3$ when $k \ge 5$.] Let $n'=n-2$ and consider the vertex set $(\Z/n'\Z)\cup \lbrace a,b\rbrace$  where $a,b\not\in\Z/n'\Z$ are two distinct symbols.

By Lemma \ref{lem: gen sequence, 1 and 2 missing}, there is a generating sequence for $n'$ such that only the residues 1 and 2 are missing and whose corresponding circular simplicial complex $\towriteornottowrite{\C}{C}$ on $n'$ vertices can be cut in such a way that the edge $\lbrace0,13\rbrace$ is destroyed and the edge $\lbrace 0,7\rbrace$ is at one end of the complex. Proposition~\ref{prop: with turn} shows that this simplicial complex is good and covers all edges in $\binom{\Z/n'\Z}{2}$ except for those whose residue is 1 or 2 and the edge $\lbrace0,13\rbrace$. Furthermore, all edges incident to $a$ or $b$ remain uncovered.  

At the edge $\lbrace 0,7\rbrace$, we attach the rotation around $a$ with the \towriteornottowrite{sequence}{path}
\[ 7,0,13,15,17,19,\dots, 4k-1,4k,4k-2,\dots,8,9,11,b \]
where the first dots indicate that we always increase by 2 until we reach $4k-1$ and the second dots indicate that we always decrease by 2 until we reach 8.
Then we have an ad-hoc part of constant size as shown in Figure \ref{fig:n=4k+3}. Finally, we attach the rotation around $b$ with the \towriteornottowrite{sequence}{path}
\[ 4k,0,4k-1,4k-2,4k-3,4k-4,\dots,12, \] where the dots indicate that we always decrease by 1 until we reach 12,
and finish with the triangle $11,12,13$.

\begin{figure}[H]
    \centering
    \begin{tikzpicture}[scale=0.8, rotate=240]
    \tikzstyle{every node}=[font=\footnotesize]
    \draw[dashed] (.5,0.86603) -- (0,0) -- (1,0);
    \draw (.5,0.86603) -- (1.5,0.86603) -- (1,0) (2,0) -- (1.5,0.86603) (2,0) -- (2.5,0.86603) -- (1.5,0.86603) -- (2,1.73306) -- (1,1.73306) -- (1.5,0.86603) (2,1.73306) -- (1.5,2.59809) -- (1,1.73306) -- (1,2.73306) -- (1.5,2.59809) (1,2.73306) -- (.5,2.59809) -- (1,1.73306) -- (0,1.73306) -- (.5,2.59809) -- (-.5,2.59809) -- (0,1.73306)  (.5,2.59809) -- (0,3.46412) -- (-.5,2.59809) -- (-.5,3.59809) -- (0,3.46412) (-.5,2.59809) -- (-1,3.46412) -- (-.5,3.59809) (-.5,2.59809) -- (-1.5,2.59809) -- (-1,3.46412) -- (-2,3.46412) -- (-1.5,2.59809) (-2,3.46412) -- (-1.5,4.33015) -- (-1,3.46412) -- (-.5,4.33015) -- (-1.5,4.33015) -- (-1,5.19618) -- (-.5,4.33015) (-1,5.19618) -- (-2,5.19618) -- (-1.5,4.33015) -- (-2.5,4.33015) -- (-2,5.19618) -- (-3,5.19618) -- (-2.5,4.33015) (-3,5.19618) -- (-2.5,6.06221) -- (-2,5.19618) -- (-1.5,6.06221) -- (-2.5,6.06221) (-1.5,6.06221) -- (-2,6.92824) -- (-2.5,6.06221) -- (-3,6.92824) -- (-2,6.92824) -- (-2.5,7.79427) -- (-3,6.92824) (-1.5,7.79427) -- (-2,6.92824) -- (-1,6.92824) -- (-1.5,7.79427) -- (-.5,7.79427) -- (-1,6.92824);
    \draw[red, thick] (.5,0.86603) -- (1,0) -- (2,0) -- (2.5,0.86603) (2,1.73306) -- (1,1.73306);
    \draw[orange, thick] (-1,6.92824) -- (-1.5,7.79427) (-2.5,7.79427) -- (-3,6.92824) -- (-2.5,6.06221) -- (-1.5,6.06221);
    \draw [red, thick, domain=0:60] plot ({1.5+cos(\x)}, {0.86603+sin(\x)});
    \draw [orange, thick, domain=120:60] plot ({-2+cos(\x)}, {6.92824+sin(\x)});
    \textat{1,-.2}{0}
    \textat{.5,1.06603}{7}
    \textat{2,-.2}{13}
    \textat{1.5,1.16603}{$a$}
    \textat{2.7,.86603}{15}
    \textat{2.2,1.93206}{11}
    \textat{1.7,2.79809}{10}
    \textat{1,3}{9}
    \textat{.6,2.79809}{7}
    \textat{.9,1.5}{$b$}
    \textat{0,1.5}{5}
    \textat{-.65,2.4}{6}
    \textat{0.1,3.68}{8}
    \textat{-.55,3.75}{$b$}
    \textat{-1.5,2.4}{$a$}
    \textat{-1.,3.1}{4}
    \textat{-2.2,3.45}{2}
    \textat{-1.5,4}{3}
    \textat{-.35,4.3}{5}
    \textat{-1.,5.4}{$a$}
    \textat{-2.,5.5}{1}
    \textat{-2.7,4.2}{$b$}
    \textat{-3.2,5.1}{2}
    \textat{-2.7,6.06}{0}
    \textat{-1.3,6.06}{$4k$}
    \textat{-3.4,7.}{$4k-1$}
    \textat{-2.,7.3}{$b$}
    \textat{-.8,6.7}{12}
    \textat{-.3,7.85}{11}
    \textat{-1.5,8.1}{13}
    \textat{-3.,8.1}{$4k-2$}
    \def\textcol{red}
    \textat{3,1.5}{$\substack{17,19,\dots, 4k-1,\\4k,4k-2,\dots,10,8,9}$}
    \def\textcol{orange}
    \textat{-3,9.5}{$4k-3,4k-4,\dots, 14$}
    \end{tikzpicture}
    \caption{The construction for $n=4k+3$}
    \label{fig:n=4k+3}
\end{figure}
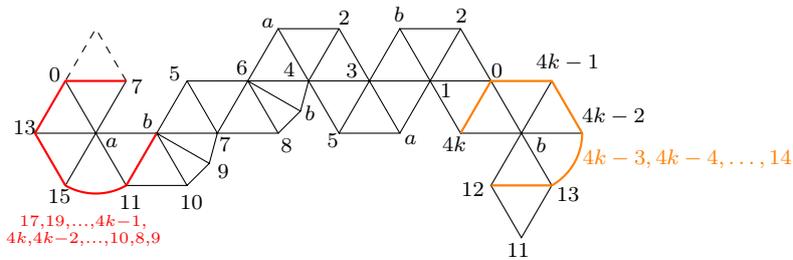
 
One must check that all edges corresponding to residues 1 or 2 as well as all edges incident to $a$ or $b$ and the edge $\lbrace0,13\rbrace$ are covered exactly once with these new triangles. In total, these are $2n'+(2n'+1)+1=4n'+2=16k+6$ edges. One can complete this check by hand by first confirming that the new triangles added exactly $16k+6$ edges and then checking that each of the edges above is covered by the new triangles at least once. We 
do this in full detail in Appendix~\ref{sec: Check for 4n+3}.

We can conclude that the simplicial complex is good and covers all $\binom n2$  edges. By  Observation~\ref{obs:sequences}, the diameter of the simplicial complex is $\tfrac{1}{2}\binom{n}{2}-\tfrac 32$.
\end{proof}

\subsection{The Construction for $n=4k+6$}\label{sec:4k+2}
The case $n=4k+6$ is the most complicated one.
Again, we have that $\binom{n}{2}$ is odd, whence we have to use every edge exactly once to get an optimal simplicial complex.
We find an optimal simplicial complex by combining the ideas of the previous two sections.
We let $n'=4k+1$ and consider the vertex set $\Z/n'\Z \cup \{ a,b,c,d,e \}$ where $a,b,c,d\not\in\Z/n'\Z$ are distinct symbols.
Theoretically, we could just remove one vertex, but we could not find suitable generating sequences for $4k+3$ with one missing residue.
Instead, we find a generating sequence for $n'=4k+1$ which misses the residues $1$, $2$, $4$, and $8$.
This defines a circular simplicial complex which we cut at a certain position.
Contrary to the previous sections, we now add triangles on \emph{both} ends of the simplicial complex.
On one end, we use the residues $1$ and $2$ to cover the edges incident to $a$ or $b$ with rotations, quite similar to the case $n=4k+3$.
On the other end, we use the residues $4$ and $8$ to cover the edges incident to $c$, $d$, and $e$ which is quite similar to the case $4k+4$, but with some additional complications.

To be able to do it, we need a generating sequence that not only misses the residues 1, 2, 4, and 8 but also can be cut in such a way that we have control over what edge we cut and which edge appears at both ends of the simplicial complex. The next lemma shows that such a generating sequence exists.

\begin{lemma}\label{lem: gen sequence, 1,2,4,8 missing}
    For every $k$ with $k\ge 7$, there exists a generating sequence for $n'=4k+1$ with the following properties:
    \begin{itemize}
        \item Exactly the residues 1, 2, 4, and 8 are missing.
        \item It is possible to cut the corresponding circular simplicial complex at $\lbrace 0,17 \rbrace$ such that the edge $\lbrace 6, 17\rbrace$ is at one end of the simplicial complex and $\lbrace0,6\rbrace$ is at the other end.
    \end{itemize}
\end{lemma}

\begin{proof}
    For $n'=29$, the sequence \[5\upper{10}\hat{14}\upper{7}3\upper96\upper{12}11\upper{13}\]
    works, whereas for $n'=33$, we can pick the sequence
     \[3\upper{12}9\upper{14}\hat{5}\upper{13}6\upper{16}11\upper{15}7\upper{10}.\]
     Note that the fact that the 11 appears directly after the 6 in the black sequence ensures that in the circular simplicial complex, we have a triangle with the vertices $\lbrace 0,6,17\rbrace$. There, we can cut $\lbrace0,17\rbrace$ to get a simplicial complex required by the second property of the lemma.

     For larger $k$, we again distinguish between $k$ odd and $k$ even. 
     
     \textit{Case 1.} Let $k=2 \ell+1$ be odd with $\ell\ge 4$. We claim that the following sequence works for $n'=8\ell+5$:
     \[7\upper{12}5\upper{4\ell-1}(4\ell-6)\upper{4\ell-2}13\upper{3}\hat{-16}\upper{-9}(-6)\upper{11}\hat{17}\upper{21}10\upper{25}15\upper{29}14\upper{33}19\dots(4\ell-10)\upper{-20}(4\ell-5)\upper{4\ell+2}\]
     Starting with $10,15,14,9,\dots$, the subsequent elements of the black sequence are always created by alternately adding 5 and subtracting 1. For $\ell=4$, this sequence just becomes \[7\upper{12}5\upper{15}10\upper{14}13\upper{3}\hat{-16}\upper{-9}(-6)\upper{11}\hat{17}\upper{18}\]
     which is a generating sequence. 
     
     For $\ell\ge 5$, we first check that all residues except for 1, 2, 4, and 8 are covered by the sequence. The black sequence covers almost all numbers which are $3$ mod $4$ via $15,19,\dots,4\ell-5$. The residues 3, 7, 11, and $4\ell-1$ appear at the beginning of the sequence. Furthermore, the black sequence covers almost every number which is $2$ mod $4$ via $10,14,\dots,4\ell-10$. The residues $6, 4\ell-6$ and $4\ell-2$ appear at the beginning of the sequence whereas $4\ell +2$ appears at the end of the sequence. The blue numbers starting with $21,25,\dots,4\ell+1$ cover almost all residues that are $1$ mod 4 and finish with $4\ell+5=-(4\ell), 4\ell+9=-(4\ell-4),\dots,8\ell -15=-20$ covering almost  all residues that are $0$ mod 4. The residues $5$, $9$, 13, 17, 16, and 12 can be found at the beginning of the sequence. This shows \ref{enum: no residue twice} of Definition \ref{def: generating sequence}. 
     
     Furthermore, \ref{enum: no consecutive turns} of Definition \ref{def: generating sequence}  is satisfied. The sum of the black sequence is 
    \begin{align*}
        7+5+(4\ell-6)+13-16-6+17+\sum_{i=1}^{\ell-4}(4i+6)+\sum_{i=1}^{\ell-4}(4i+11)&=4\ell+14+17\ell-68+4\ell^2-28\ell+48\\
        &=4\ell^2-7\ell-6
    \end{align*}
    and the gcd is
    \begin{equation*}
\left(4\ell^2-7\ell-6,8\ell+5\right)=\left(8\ell^2-14\ell-12,8\ell+5\right)=(-19\ell-12,8\ell+5)=(5\ell+3,8\ell+5)=1.
    \end{equation*}
    Thus, all conditions of Definition \ref{def: generating sequence} are fulfilled and this is indeed a generating sequence.  
    
    Finally, we have to check the second condition of the lemma. Figure~\ref{fig: cut possible} shows a part of the corresponding circular simplicial complex.
    It can be seen that $0,6,17$ is a triangle and that if the edge $\lbrace0,17\rbrace$ is deleted, the simplicial complex has two ends with edge $\lbrace 0,6\rbrace$ at one of them and $\lbrace6,17\rbrace$ at the other one.
    This completes the proof for odd $k$.

        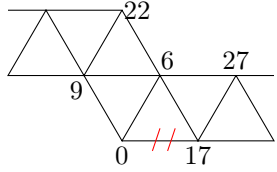
\begin{figure}[H]
        \centering
        \begin{tikzpicture}
        \draw (0,.86603) -- (.5,.86603) -- (0,0) -- (1,0) -- (.5,.86603) -- (1.5,.86603) -- (1,0) -- (2,0) -- (1.5,.86603) (1,0) -- (1.5,-.86603) -- (2,0) -- (2.5,-.86603) -- (1.5,-.86603) (2,0) -- (3,0) -- (2.5,-.86603) -- (3.5,-.86603) -- (3,0) -- (3.5,0);
        \textat{1.7,.86603}{22}
        \textat{2.1,.2}{6}
        \textat{1.5,-1.06603}{0}
        \textat{2.5,-1.06603}{17}
        \textat{3,0.2}{27}
        \textat{.9,-.2}{9}
        \draw[red] (1.9,-1) -- (2,-0.7) (2.1,-1) -- (2.2,-0.7);
        \end{tikzpicture}
        \caption{circular simplicial complex corresponding to the generating sequence for odd $k$. It can be cut as the lemma requires.}
        \label{fig: cut possible}
    \end{figure}

\textit{Case 2.} Let $k=2 \ell$ be even with $\ell\ge5$. We claim that the following sequence works for $n'=8\ell+1$:
    \[(-4\ell+1)\upper{4\ell-6}\widehat 5\upper{12}(-4\ell+5)\upper{4\ell-2}-3\upper{16}-13\upper{7}6\upper{17}11\upper{21}10\upper{25}15\dots (4\ell-14)\upper{-24}(4\ell-9)\upper{-20}(4\ell-10)\upper{-9}\]
    Starting with $6,11,10,15,\dots$, the subsequent elements of the black sequence are always created by alternately adding 5 and subtracting 1.
    
    We check that it covers all residues except for 1, 2, 4, and 8. The black sequence covers almost all numbers which are $3$ mod $4$ via $11,15,\dots,4\ell-9$. The residues 3, 7, $4\ell-5$ and $4\ell-1$ appear at the beginning of the sequence. Furthermore, the black sequence covers almost every number which is $2$ mod $4$ via $6,10,\dots,4\ell-10$. The residues $4\ell-6$ and $4\ell-2$ appear at the beginning of the sequence. The blue numbers starting with $17,21,\dots,4\ell-3$ cover almost all residues that are $1$ mod 4 and finish with $4\ell+1=-(4\ell), 4\ell+5=-(4\ell-4),\dots,8\ell -19=-20$ covering almost  all residues that are $0$ mod 4. The residues $5$, 13, 16, and 12 can be found at the beginning of the sequence whereas 9 can be found at the end of the sequence. This shows \ref{enum: no residue twice} of Definition \ref{def: generating sequence}. \ref{enum: no consecutive turns} is satisfied because $\abs I = 1$.
    
    The black sequence has the sum
\begin{align*}
    (-4\ell+1)+5+(-4\ell+5)-3-13+6+\sum_{i=1}^{\ell-4}(4i+7)+\sum_{i=1}^{\ell-4}(4i+6)&=1-8\ell+13\ell-52+4\ell^2-28\ell+48\\
    &=4\ell^2-23\ell-3
\end{align*}
and the gcd is
\begin{equation*}\left(4\ell^2-23\ell-3,8\ell+1\right)=\left(8\ell^2-46\ell-6,8\ell+1\right)=(-47\ell-6,8\ell+1)=(\ell, 8\ell+1)=1.
\end{equation*}
Hence, it is a generating sequence.

Furthermore, the second condition of the lemma is fulfilled because the 11 appears directly after 6 in the sequence.
\end{proof}
\towriteornottowrite{
\begin{definition}
    For a $k\ge 7$, let $\C_k$ be the simplicial complex that we get after cutting the circular simplicial complex corresponding to the generating sequence given by Lemma~\ref{lem: gen sequence, 1,2,4,8 missing} at the edge $\lbrace0,17\rbrace$.
\end{definition}
Note that $\C_k$ covers all edges on the vertex set $\Z/n'\Z\cup\{a,b,c,d,e\}$ except the following:}{
Let $C$ be the simplicial complex that we get after we cut the corresponding circular simplicial complex of the generating sequence of the previous lemma at the edge $\lbrace0,17\rbrace$. We still have to cover the following edges:}
\begin{itemize}
    \item The edges corresponding to the residues 1, 2, 4 or 8.
    \item The edges incident to $a$, $b$, $c$, $d$ or $e$.
    \item The edge $\lbrace0,17\rbrace$.
\end{itemize}
\towriteornottowrite{In the next lemma, we will attach triangles to $\C_k$ to cover roughly half of these edges.}{}
\begin{lemma}\label{lem: n=4k+2 first side}\towriteornottowrite{For every $k\ge 7$ with $n'\coleq 4k+1$, i}{I}t is possible to attach triangles to $\towriteornottowrite{\C_k}{C}$ at the edge $\lbrace6,17\rbrace$ such that exactly the following edges are covered:
\begin{itemize}
    \item The edges corresponding to the residues 1 and 2.
    \item The edges incident to $a$ or $b$ except for the ones that are also incident to $c$, $d$ or $e$.
    \item The edge $\lbrace0,17\rbrace$.
\end{itemize}
\end{lemma}
\begin{proof}
Figure~\ref{fig:n=4k+2} gives the construction.

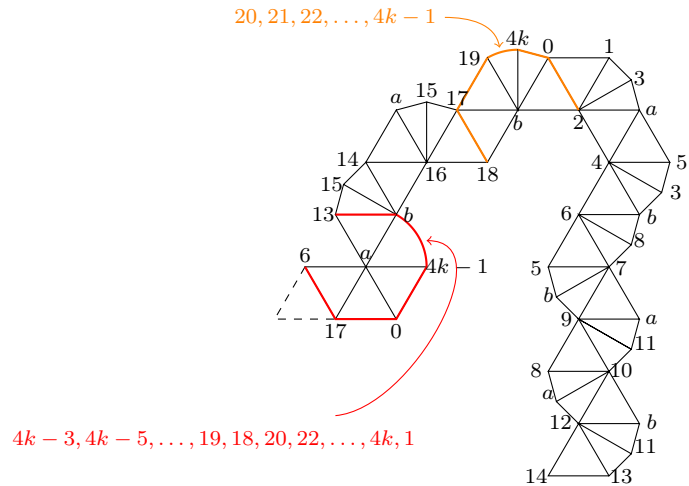
\begin{figure}[H]
    \centering
    \begin{tikzpicture}[scale=0.8]
    \tikzstyle{every node}=[font=\footnotesize]
    \draw[dashed] (.5,0.86603) -- (0,0) -- (1,0);
    \draw (1,0) -- (.5,0.86603) -- (1.5,0.86603) -- (1,0) -- (2,0) -- (1.5,0.86603) -- (2.5,0.86603) -- (2,0);
    \draw [red, thick, domain=0:60] plot ({1.5+cos(\x)}, {0.86603+sin(\x)});
    \draw [orange, thick, domain=90:120] plot ({4+cos(\x)}, {3.46412+sin(\x)});
    \draw (2,1.73206) -- (1.5,0.86603) -- (1,1.73206) -- (2,1.73206) -- (1.5,2.59809) -- (1.13398,2.23206) -- (1,1.73206) (1.13398,2.23206) -- (2,1.73206) -- (2.5,2.59809) -- (1.5,2.59809) -- (2,3.46412) -- (2.5,2.59809) -- (2.5,3.59809) -- (2,3.46412) (2.5,3.59809) -- (3,3.46412) -- (2.5,2.59809) -- (3.5,2.59809) -- (3,3.46412) -- (4,3.46412) -- (3.5,2.59809) (3,3.46412) -- (3.5,4.33015) -- (4,3.46412) -- (4,4.46412) -- (4.5,4.33015) -- (4,3.46412) -- (5,3.46412) -- (4.5,4.33015) -- (5.5,4.33015) -- (5,3.46412) -- (6,3.46412) -- (5.86603,3.96412) -- (5,3.46412) (5.86603,3.96412) -- (5.5,4.33015) (6,3.46412) -- (5.5,2.59809) -- (5,3.46412) (6,3.46412) -- (6.5,2.59809) -- (5.5,2.59809) -- (6.36603,2.09809) -- (6.5,2.59809) (6.36603,2.09809) -- (6,1.73206) -- (5.5,2.59809) -- (5,1.73206) -- (6,1.73206) -- (5.86603,1.23206) -- (5,1.73206) -- (5.5, 0.86603) -- (5.86603,1.23206) (5.5, 0.86603) -- (4.5, 0.86603) -- (5,1.73206) (4.5, 0.86603) -- (4.63397,0.36603) -- (5.5, 0.86603) (4.63397,0.36603) -- (5,0) -- (5.5, 0.86603) -- (6,0) -- (5,0) -- (5.86603, -0.5) -- (6,0) (5.86603, -0.5) -- (5,0) -- (5.5, -0.86603) -- (5.86603, -0.5) (5.5, -0.86603) -- (4.5, -0.86603) -- (5,0) (4.5, -0.86603) -- (4.63397, -1.36603) -- (5.5, -0.86603) (4.63397, -1.36603) -- (5,-1.73206) -- (5.5, -0.86603) -- (6,-1.73206) -- (5,-1.73206) -- (5.86603,-2.23206) -- (6,-1.73206) (5.86603,-2.23206) --  (5.5,-2.59809) -- (5,-1.73206) -- (4.5,-2.59809) -- (5.5,-2.59809);
    \draw [red, thick] (.5,0.86603) -- (1,0) -- (2,0) -- (2.5,.86603) (2,1.73206) -- (1,1.73206);
    \draw [orange, thick] (3.5,2.59809) -- (3,3.46412) -- (3.5,4.33015) (4,4.46412) -- (4.5,4.33015) -- (5,3.46412);
    \textat{.5,1.06603}{6}
    \textat{1,-.2}{17}
    \textat{2,-.2}{0}
    \textat{1.5,1.06603}{$a$}
    \textat{3.,.86603}{$4k-1$}
    \textat{2.2,1.73206}{$b$}
    \textat{0.8,1.73206}{13}
    \textat{0.93398,2.23206}{15}
    \textat{1.2,2.64809}{14}
    \textat{2.65,2.4}{16}
    \textat{2,3.66412}{$a$}
    \textat{2.5,3.8}{15}
    \textat{3,3.66412}{17}
    \textat{3.5,2.4}{18}
    \textat{3.2,4.33015}{19}
    \textat{4,4.66412}{$4k$}
    \textat{4,3.26412}{$b$}
    \textat{4.5,4.53015}{0}
    \textat{5,3.26412}{2}
    \textat{5.5,4.53015}{1}
    \textat{6.,4}{3}
    \textat{6.2,3.46412}{$a$}
    \textat{5.3,2.6}{4}
    \textat{6.7,2.6}{5}
    \textat{6.6,2.1}{3}
    \textat{6.2,1.7}{$b$}
    \textat{6.,1.3}{8}
    \textat{5.7,0.85}{7}
    \textat{4.8,1.75}{6}
    \textat{4.3,0.85}{5}
    \textat{4.5,0.4}{$b$}
    \textat{4.8,0}{9}
    \textat{6.2,0}{$a$}
    \textat{6.1,-0.4}{11}
    \textat{5.7,-.85}{10}
    \textat{4.3,-.85}{8}
    \textat{4.5,-1.25}{$a$}
    \textat{4.7,-1.7}{12}
    \textat{6.2,-1.7}{$b$}
    \textat{6.1,-2.1}{11}
    \textat{5.7,-2.6}{13}
    \textat{4.3,-2.6}{14}
    \def\textcol{red}
    \textat{-1,-2}{$4k-3,4k-5,\dots,19,18,20,22,\dots,4k,1$}
    \def\textcol{orange}
    \textat{1,5}{$20,21,22,\dots,4k-1$}
    \draw [->, red] (1,-1.6) to [out = 5, in = 0] (2.5,1.3);
    \draw [->, orange] (2.8,5) to [out = 0, in = 100] (3.7,4.5);
    \end{tikzpicture}
    \caption{The construction for $n=4k+2$ covering $a$, $b$ and the residues 1 and 2.}
    \label{fig:n=4k+2}
\end{figure}

It contains two rotations\towriteornottowrite{}{ turns}, one around $a$ with the \towriteornottowrite{sequence}{path}
\[ 6,17,0,4k-1,4k-3,4k-5,\dots,19,18,20,22,\dots,4k,1,b,13, \]
where the dots indicate that we always decrease by 2 until we reach 19,
and one around $b$ with the \towriteornottowrite{sequence}{path}
\[ 18,17,19,20,21,22,\dots,4k,0,2, \]
where the dots indicate that we always increase by 1 until we reach $4k$.

Additionally, there are some constant-sized constructions that ensure that exactly the requested edges are covered.

In total, in this lemma, we have to cover $2n'+(2n'+1)+1=4n'+2=16k+6$ edges. To complete the proof of this lemma, one has to check by hand that exactly $16k+6$ edges were added in Figure~\ref{fig:n=4k+2} and that each of the edges listed in the lemma was added at least once.

We do this in Appendix~\ref{sec: Check for 4n+2 first side}.
\end{proof}
\towriteornottowrite{In the next lemma, we will cover the remaining uncovered edges by attaching triangles to the other side of $\C_k$.}{}
\begin{lemma}\label{lem: n=4k+2 second side}
    \towriteornottowrite{For every $k\ge 7$ with $n'\coleq 4k+1$, i}{I}t is possible to attach triangles to $\towriteornottowrite{\C_k}{C}$ at the edge $\lbrace0,6\rbrace$ such that exactly the following edges are covered:
\begin{itemize}
    \item The edges corresponding to the residues 4 and 8.
    \item The edges incident to $c$, $d$, or $e$.
\end{itemize}
\end{lemma}
\begin{proof}
Figure~\ref{fig:n=4k+3 second half} shows the construction. It distinguishes between $k$ being odd and $k$ being even.

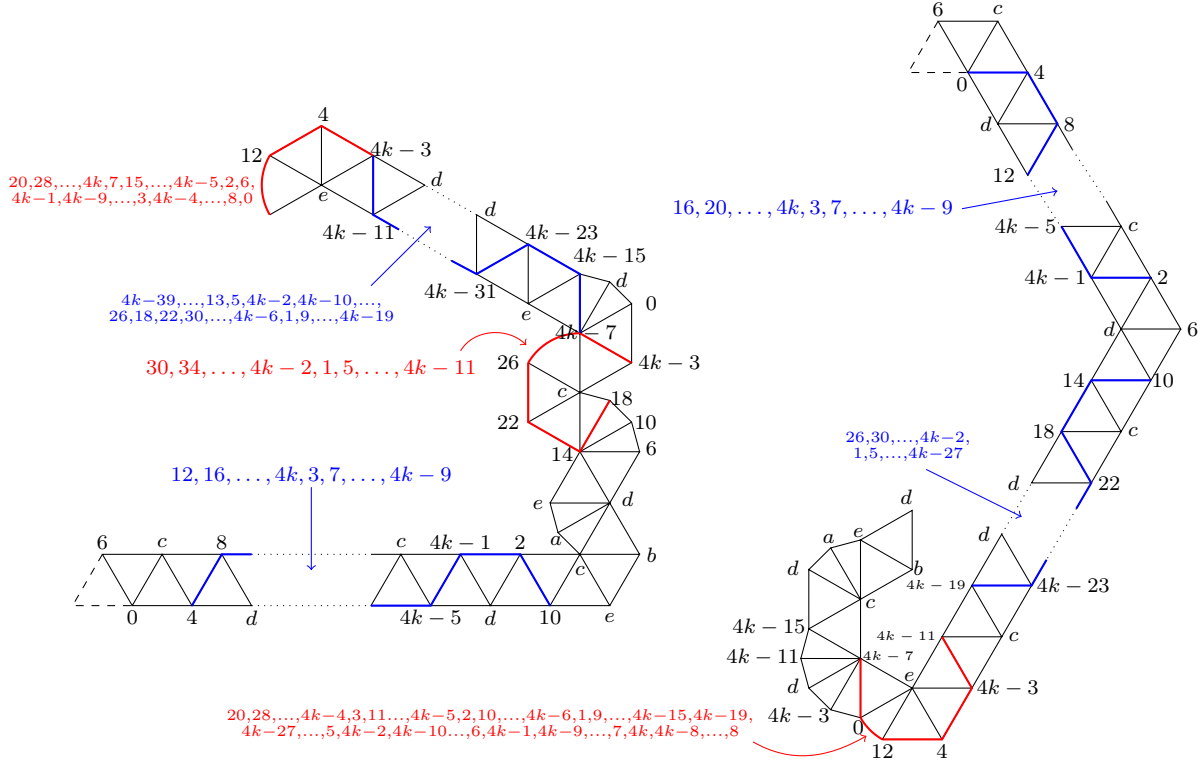
\begin{figure}[H]
    \centering
    \resizebox{\linewidth}{!}{
    \begin{tikzpicture}[scale=0.8]
    \tikzstyle{every node}=[font=\footnotesize]
    \draw (1,0) -- (.5,0.86603) -- (1.5,0.86603) -- (1,0) -- (2,0) -- (1.5,0.86603) -- (2.5,0.86603) -- (2,0) -- (3,0) -- (2.5,0.86603) -- (3,0.86603) (5,0.86603) -- (5.5,0.86603) -- (5,0) -- (6,0) -- (5.5,0.86603) -- (6.5,0.86603) -- (6,0) -- (7,0) -- (6.5,0.86603) -- (7.5,0.86603) -- (7,0) -- (8,0) -- (7.5,0.86603) -- (8.5,0.86603) -- (8,0) -- (9,0) -- (8.5,0.86603) -- (9.5,0.86603) -- (9,0) (9.5,0.86603) -- (9,1.73206) -- (8.5,0.86603) -- (8.13397,1.23206) -- (9,1.73206) (8.13397,1.23206) -- (8,1.73206) -- (9,1.73206) -- (8.5,2.59809) -- (8,1.73206) (8.5,2.59809) -- (9.5,2.59809) -- (9,1.73206) (9.5,2.59809) -- (9.36603,3.09809); 
    \draw[dashed] (.5,0.86603) -- (0,0) -- (1,0);
    \draw[dotted] (3,0) -- (5,0) (3,.86603) -- (5,.86603);
    \draw[blue, thick] (2,0) -- (2.5,0.86603) -- (3,0.86603) (5,0) -- (6,0) -- (6.5,0.86603) -- (7.5,0.86603) -- (8,0);
    
    \begin{scope}[shift = {(-1,0)}]
    \begin{scope}[rotate around={30:(9.5,2.59809)}]
    \draw[] (9.5,2.59809) -- (10.5,2.59809)(10.36603,3.09809) -- (10.5,2.59809) (9.5,2.59809) -- (10.36603,3.09809) -- (10,3.46412) -- (9.5,2.59809) -- (9,3.46412) -- (10,3.46412) -- (9.5,4.33015) -- (9,3.46412) (10,3.46412) -- (10.5,4.33015) -- (11,3.46412) -- (10,3.46412) (10.5,4.33015) -- (11.5,4.33015) -- (11,3.46412) (11.5,4.33015) -- (11.36603,4.83015) -- (10.5,4.33015) -- (11,5.19618) -- (11.36603,4.83015) (11,5.19618) -- (10,5.19618) -- (10.5,4.33015) (11,5.19618) -- (10.5, 6.06221) -- (10,5.19618) -- (9.5, 6.06221) -- (10.5, 6.06221) -- (10, 6.92824) -- (9.5, 6.06221) -- (9.25,6.465) (8.75,7.36125) -- (8.5,7.79427) -- (9.5, 7.79427) -- (9,8.6603) -- (8.5,7.79427) -- (8,8.66029) -- (9,8.66029) -- (8.5,9.52632) -- (8,8.66026) -- (7.5,9.52632) -- (8.5,9.52632) (7,8.66029) -- (8,8.66029);
    \draw[red, thick](10.5,4.33015) -- (11,3.46412) (9.5,4.33015)-- (9,3.46412)-- (9.5,2.59809) -- (10.36603,3.09809);
    
    \draw[blue, thick] (10.5,4.33015) -- (11,5.19618) -- (10.5, 6.06221) -- (9.5, 6.06221) -- (9.25,6.465) (8.75,7.36125) -- (8.5,7.79427) -- (9,8.6603);
    \draw[dotted] (9.25,6.465) -- (8.75,7.36125) (10, 6.92824) -- (9.5, 7.79427);
    \draw[red, thick] (9,8.6603) -- (8.5,9.52632)-- (7.5,9.52632);
    \draw [red, thick, domain=120:180] plot ({8+cos(\x)}, {8.66029+sin(\x)});
    \draw [red, thick, domain=60:120] plot ({10+cos(\x)}, {3.46412+sin(\x)});
    \end{scope}
    \end{scope}

    \textat{0.5,1.066}{6}
    \textat{1.5,1.066}{$c$}
    \textat{2.5,1.066}{8}
    \textat{5.5,1.066}{$c$}
    \textat{6.5,1.066}{$4k-1$}
    \textat{7.5,1.066}{2}
    \textat{1,-0.2}{0}
    \textat{2,-0.2}{4}
    \textat{3,-0.2}{$d$}
    \textat{6,-0.2}{$4k-5$}
    \textat{7,-0.2}{$d$}
    \textat{8,-0.2}{$10$}
    \textat{9,-0.2}{$e$}
    \textat{8.5,.6}{$c$}
    \textat{9.7,0.866}{$b$}
    \textat{9.3,1.8}{$d$}
    \textat{8.1,1.1}{$a$}
    \textat{7.8,1.7}{$e$}
    \textat{8.2,2.55}{14}
    \textat{9.7,2.65}{6}
    \textat{9.6,3.1}{10}
    \textat{9.2,3.5}{18}
    \textat{8.2,3.6}{$c$}
    \textat{7.3,3.1}{22} 
    \textat{7.3,4.1}{26}
    \textat{8.6,4.6}{$4k-7$}
    \textat{10,4.1}{$4k-3$}
    \textat{9.7,5.1}{0}
    \textat{9.2,5.5}{$d$}
    \textat{9,5.9}{$4k-15$}
    \textat{8.2,6.3}{$4k-23$}
    \textat{7.,6.7}{$d$}
    \textat{6.1,7.2}{$d$}
    \textat{5.5,7.7}{$4k-3$}
    \textat{4.8,6.3}{$4k-11$}
    \textat{4.2,6.9}{$e$}
    \textat{4.2,8.3}{4}
    \textat{3.,7.6}{12}
    \textat{7.6,4.9}{$e$}
    \textat{6.5,5.3}{$4k-31$}
    {\def\textcol{blue}
    \textat{4,2.2}{$12,16,\dots,4k,3,7,\dots,4k-9$}
    \textat{3,5}{$\substack{4k-39,\dots,13,5,4k-2,4k-10,\dots,\\  26,18,22,30,\dots,4k-6,1,9,\dots,4k-19}$}
    \draw[->, blue] (4,2) -- (4,0.6);
    \draw[->, blue] (5,5.4) -- (6,6.4);}
    \draw[->, red] (6.5,4.3) to[out=60, in=135] (7.6,4.4);
    {
    \def\textcol{red}
    \textat{4,4}{$30,34,\dots,4k-2,1,5,\dots,4k-11$}
    \textat{1,7}{$\substack{20,28,\dots,4k,7,15,\dots,4k-5,2,6,\\4k-1,4k-9,\dots,3,4k-4,\dots,8,0}$}
    }

    \begin{scope}[shift={(14,9)}]
    \draw[dashed] (.5,0.86603) -- (0,0) -- (1,0);
    \draw (1,0) -- (.5,0.86603) -- (1.5,0.86603) -- (1,0) -- (2,0) -- (1.5,0.86603) (1,0) -- (1.5,-.86603) -- (2,0) -- (2.5,-.86603) -- (1.5,-.86603) -- (2,-1.732) -- (2.5,-.86603) -- (2.75,-1.299);
    \draw[blue, thick] (1,0) -- (2,0) -- (2.5,-.86603) -- (2,-1.732);
    \draw[dotted] (2,-1.732) -- (2.55,-2.5923)(2.75,-1.299) --({-0.433015+3.75},{.86603 -3.032});
    \textat{.9,-0.2}{0}
    \textat{.5,1.06}{6}
    \textat{1.5,1.06}{$c$}
    \textat{2.2,0}{4}
    \textat{1.3,-.866}{$d$}
    \textat{2.7,-.866}{8}
    \textat{1.6,-1.732}{12}
    {\def\textcol{blue}
    \textat{-1.6,-2.3}{$16,20,\dots,4k,3,7,\dots,4k-9$}
    \draw[->, blue](.9,-2.3) -- (2.5,-2);
    \textat{-0,-6.3}{$\substack{26,30,\dots,4k-2,\\1,5,\dots,4k-27}$}
    \draw[->, blue](.3,-6.7) -- (1.9,-7.5);}
    {\def\textcol{red}
    \textat{-7,-11}{$\substack{20,28,\dots,4k-4,3,11\dots,4k-5,2,10,\dots,4k-6,1,9,\dots,4k-15,4k-19,\\4k-27,\dots,5,4k-2,4k-10\dots,6,4k-1,4k-9,\dots,7,4k,4k-8,\dots,8}$}
    \draw[red,->] (-2.6,-11.3) to[out=330, in=210] (-0.7,-11.2);
    }
    \begin{scope}[shift={(-0.433015,.86603)}]
    \draw (3.75,-3.032) -- (4,-3.46412) -- (3,-3.46412) -- (3.5,-4.33015) -- (4,-3.4612) -- (4.5,-4.33015) -- (3.5,-4.33015) -- (4,-5.19618) -- (4.5,-4.33015) -- (5,-5.19618) -- (4,-5.19618) -- (4.5,-6.06221) -- (5,-5.19618) (4,-5.19618) -- (3.5,-6.06221) -- (4.5,-6.06221);
    \draw[blue, thick] (3,-3.46412) -- (3.5,-4.33015) -- (4.5,-4.33015);
    \textat{4.2,-3.46412}{$c$}
    \textat{2.4,-3.46412}{$4k-5$}
    \textat{2.9,-4.33015}{$4k-1$}
    \textat{4.7,-4.33015}{2}
    \textat{5.2,-5.19618}{6}
    \textat{3.8,-5.19618}{$d$}
    \textat{4.7,-6.06221}{10}
    \textat{3.2,-6.06221}{14}
    \begin{scope}[shift = {(4.5,-6.06221)}]
    \draw (0,0) -- (-.5,-.86603) -- (-1,0) -- (-1.5,-.86603) -- (-.5,-.86603) -- (-1,-1.732) -- (-1.5,-.86603) -- (-2,-1.732) -- (-1,-1.732) -- (-1.25,-2.165);
    \draw[blue, thick] (0,0) -- (-1,0) -- (-1.5,-.86603) -- (-1,-1.732) -- (-1.25,-2.165);
    \draw[dotted] (-1.25,-2.165) -- (-1.75,-3.031);
    \textat{-.3,-.86603}{$c$}
    \textat{-1.8,-.86603}{18}
    \textat{-.7,-1.732}{22}
    \textat{-2.3,-1.732}{$d$}
    \begin{scope}[shift = {(-2,-3.46412)}]
        \draw (.25,.433015) -- (0,0) -- (-.5,.86603) -- (-1,0) -- (0,0) -- (-.5,-.86603) -- (-1,0) -- (-1.5,-.86603) -- (-.5,-.86603) -- (-1,-1.732) -- (-1.5,-.86603) -- (-2,-1.732) -- (-1,-1.732) -- (-1.5,-2.598) -- (-2,-1.732) -- (-2.5,-2.598) -- (-1.5,-2.598);
        \draw[blue, thick] (.25,.433015) -- (0,0) -- (-1,0);
        \draw[dotted] (-.5,.86603) -- (0,1.732);
        \draw[red, thick] (-1.5,-.86603) -- (-1,-1.732) -- (-1.5, -2.598) -- (-2.5, -2.598);
    \textat{.7,0}{$4k-23$}
    \textat{-.8,.86603}{$d$}
    \textat{-1.6,0}{\tiny{$4k-19$}}
    \textat{-2.1,-.86603}{\tiny{$4k-11$}}
    \textat{-0.3,-.86603}{$c$}
    \textat{-.4,-1.732}{$4k-3$}
    \textat{-1.5,-2.798}{$4$}
    \textat{-2.5,-2.798}{$12$}
    \textat{-2.05,-1.532}{$e$}
    \textat{-2.9,-2.4}{$0$}
    \textat{-3.9,-2.2}{$4k-3$}
    \textat{-4,-1.7}{$d$}
    \textat{-4.5,-1.2}{$4k-11$}
    \textat{-4.4,-.7}{$4k-15$}
    \textat{-4.,.3}{$d$}
    \textat{-3.4,.8}{$a$}
    \textat{-2.9,.9}{$e$}
    \textat{-2.1,1.5}{$d$}
    \textat{-1.9,.3}{$b$}
    \textat{-2.7,-.3}{$c$}
    \textat{-2.4,-1.2}{\tiny{$4k-7$}}
        \begin{scope}[rotate around={30:(-2,-1.732)}]
        \draw(-2,-1.732) -- (-3,-1.732) -- (-2.5,-.86603) -- (-2,-1.732)(-2.5,-.86603) -- (-3.366,-1.36598) -- (-3,-1.732)(-3.366,-1.36598) -- (-3.5,-.86693) -- (-2.5,-.86693) -- (-3.36597,-0.36603) -- (-3.5,-.86693) (-3.36597,-0.36603) -- (-3,0) -- (-2.5,-.86693) -- (-2, 0) -- (-3,0) -- (-2.5,.86693) -- (-2,0) -- (-1.5,.86693) -- (-2,1) -- (-2,0) (-2,1) -- (-2.5,.86693) (-2,0) -- (-1,0) -- (-1.5,.86693) -- (-.5,.86693) -- (-1,0);
        \draw[thick, red](-3,-1.732) -- (-2.5,-.86603);

    \end{scope}
    \draw [red, thick, domain=210:240] plot ({-2+cos(\x)}, {-1.732+sin(\x)});
    \end{scope}
    \end{scope}
    \end{scope}
    \end{scope}
    \end{tikzpicture}}
    \caption{The construction for $n=4k+2$ covering $c$, $d$ and $e$ together with the residues 4 and 8. The left construction holds for odd $k$ and the right construction for even $k$.}
    \label{fig:n=4k+3 second half}
\end{figure}

We start with even $k$. There we have a zig-zag around $(c,d)$ with \towriteornottowrite{sequence}{path}
\[ 0,4,8,12,16,20,\dots,4k,3,7,\dots,4k-1,2 \]
where the dots indicate that we always increase by 4 until we reach $4k$ resp.\@ $4k1-$.
Then there is a zig-zag around $(d,c)$ with \towriteornottowrite{sequence}{path}
 \[ 10,14,18,22,26,30,\dots,4k-2,1,5,\dots,4k-19 \, , \]
 again with the dots indicating that we always increase by 4 until we reach $4k-2$ resp.\@ $4k-19$.
These two zig-zags could be seen as a single zig-zag, but there is a turn in the middle.
We also have a rotation around $e$  \towriteornottowrite{with the sequence}{along the path}
\begin{align*}
    4k-11,4k-3,4,12,20,28,\dots,4k-4,3,11,\dots,4k-5,2,10,\dots,4k-6,1,9,\dots,4k-15,4k-19,\\
    4k-27,\dots,5,4k-2,4k-10,\dots,6,4k-1,4k-9,\dots,7,4k,4k-8,\dots,8,0,4k-7 \,
\end{align*}
where the dots always indicate that we increase by 8. 
The zig-zags mostly use residue $4$ and the rotation mostly uses residue $8$.

For odd $k$, the complex contains a zig-zag around $(c,d)$ \towriteornottowrite{with the sequence}{along the path}
\[ 4,8,12,16,\dots,4k,3,7,\dots,4k-1,2,10 \, , \]
where the dots always indicate an increase by 4, a rotation around $c$ \towriteornottowrite{with the sequence}{along the path}
 \[ 18,14,22,26,30,34,\dots,4k-2,1,5,\dots,4k-3 \, , \]
 where the dots always indicate an increase of 4,
a zig-zag around $(d,e)$ \towriteornottowrite{with the sequence}{along the path}
\begin{align*}
    4k-7,4k-15,4k-23,4k-31,4k-39,\dots,5,4k-2,4k-10,\dots,18,22,30,\dots,4k-6,1,9,\dots,4k-3 \, ,
\end{align*}
where the dots always indicate an increase resp.\@ decrease by 8,
 and a rotation around $e$ \towriteornottowrite{with the sequence}{along the path}
 \[ 4k-3,4,12,20,28,\dots,4k,7,15,\dots,4k-5,2,6,4k-1,4k-9,\dots,3,4k-4,\dots,8,0 \, ,\]
 where the dots always indicate an increase resp.\@ decrease by 8,
as shown in Figure~\ref{fig:n=4k+3 second half}.
Note that this time the rotations around $c$ and $e$ are shorter as they are both also used in zig-zags together with $d$.
Otherwise, the main thing to note is that the rotation around $c$ and the zig-zag involving $c$ mostly cover the residue $4$, while the rotation around $e$ and the zig-zag involving $e$ mostly cover the residue $8$.

In both cases, the lemma states that we have to cover exactly $2n'+(3n'+3)+6=5n'+9=20k+14$ edges. To finish the proof, we have to check for both constructions that they contain exactly $20k+14$ edges and that each edge listed in the lemma appears in the construction at least once. We do this in Appendix~\ref{sec: Check for 4k+2 second side}.

This completes the proof.
\end{proof}

The last three lemmas together with Proposition~\ref{prop: with turn} prove that for $n=4k+6$ and $k\ge 7$, there is a good simplicial complex that covers all $\binom{n}{2}$ edges. By Observation~\ref{obs:sequences}, this simplicial complex has a diameter of $\tfrac{1}{2}\binom{n}{2}-\tfrac 32$. This proves Theorem~\ref{thm:main, 4k+1} for the case $n=4k+6$ of for $k \ge 7$.

In summary, in Section~\ref{sec:Thm1.1} we determined $H_s(n,2)$ for all $n > 30$ and $n \in \{ 13,17,20,21,23,24,25,27,28,29 \}$.
In Appendix~\ref{sec:appendix}, we describe optimal simplicial complexes for all small $n$ for which this section did not work.
All together, this proves Theorem~\ref{thm:main, 4k+1}.

\section{Proof of Theorem~\ref{thm:packHC2}}
\label{sec:proofHC2}

\begin{lemma}
\label{lem:order_of_2}
\towriteornottowrite{If}{Let} $p\equiv 5\mod8$ is a prime, then $\ord_p(2)$ is divisible by 4.
\end{lemma}
\begin{proof}
This is a well-known fact. If $p\equiv 5\mod8$, the Second Supplement to Law of Quadratic Reciprocity~\cite[Proposition 5.1.3]{ireland1990classical} implies that \[2^{\frac{p-1}{2}}\equiv -1 \mod p \, .\]
Hence, $\ord_p(2)$ does not divide \towriteornottowrite{the even number $\frac{p-1}{2} =: 2\ell$, yet by Lagrange's Theorem, it does divide $p-1=4\ell$}{$\frac{p-1}2$ but it divides $p-1$}. Thus, \towriteornottowrite{$\ord_p(2)$ must be}{it is} divisible by 4.
\end{proof}

Note that Lemma~\ref{lem:order_of_2} is not a characterization of those primes for which $\ord_p(2)$ is divisible by four and there are more, e.g.~17.
We are not aware of any other easy criteria that guarantee this, let alone a full characterization.

The idea in the proof of Theorem~\ref{thm:packHC2} is that the pairs $(x,2x)$ form cycles divisible by four, i.e.~we can find a matching which then forms our square of Hamilton cycles.

\begin{proof}[Proof of Theorem~\ref{thm:packHC2}]
We show that there are $\ell=\tfrac{p-1}{4}$ elements $a_1, \dots,a_\ell\in\F_p^\ast$ such that the squares of Hamilton cycles which are given by the vertex ordering \[\big(0,a_i,2a_i,3a_i,\dots, (p-1)a_i\big)\] for $i \in [\ell]$ are pairwise edge-disjoint and, therefore, partition the edges of $K_p$.

Consider the normal (multiplicative) subgroup $H=\spa 2\subs \mathbb F_p^\ast$. It partitions $\F_p^\ast$ into left cosets $a\cdot H$ with $a\in\F_p^\ast$ and each coset has a size of $\abs{a\cdot H}=\abs{H}=\ord_p(2)$ which is divisible by 4. Notice that $2^{\ord_p(2)/2}=-1$ \towriteornottowrite{in $\mathbb F_p$ because $2^{\ord_p(2)/2}$ is a root of $X^2-1$, but it cannot be 1 by the minimality of $\ord_p(2)$}{}. Thus, for $a\in \F_p^\ast$ and $k\in \lbrace 0,\dots, \ord_p(2)/2-1\rbrace$, we get \[a\cdot 2^{k+\ord_p(2)/2}=-a\cdot 2^k.\]
Let $R=\lbrace a_1,\dots, a_m\rbrace\subs \F_p^\ast$ be a set of representatives of $\F_p^\ast/H$. Then
\[\left\lbrace a\cdot 2^k, -a\cdot 2^k|a\in R, 0\le k\le\ord_p(2)/2-1\right\rbrace\]
contains each element of $\F_p^\ast$ exactly once.

Let's view the vertices of $K_p$ as elements of $\F_p$. Given an $a\in \F_p^\ast$ and a $0\le k\le \ord_p(2)/2-1$, we define the square of a Hamilton cycle $H_{a,k}$ via the vertex ordering
\[\big(0\cdot a\cdot 2^k, 1\cdot a\cdot 2^k, 2\cdot a\cdot 2^k,\dots,(p-1)\cdot a\cdot 2^k\big).\]
As $a\cdot 2^k$ is coprime to $p$, every vertex appears in this ordering exactly once. Thus, $H_{a,k}$ contains all edges $uv$ with \[u-v\in \big\lbrace a\cdot 2^k,-a\cdot 2^k, a\cdot 2^{k+1},-a\cdot 2^{k+1}\big\rbrace.\]
Therefore the collection \[\big\lbrace H_{a,k}|a\in R, 0\le k\le \ord_p(2)/2-1, 2\vert k\big\rbrace\]
contains each edge exactly once. Here we used that $\ord_p(2)$ is not only even but also divisible by 4.
\end{proof}

\section{Concluding remarks}
\label{sec:conclusion}

\paragraph{Packing squares of Hamilton cycles.} 
Question~\ref{quest:HC2} belongs to the extensive field of packing problems, where pairwise disjoint embeddings of some graph $H$ into some host graph $G$ are sought for, ideally covering each edge of $G$ exactly once. Manifestations of this general problem have long been studied, for a wide range of graphs $H$ and host graphs $G$, as well as for hypergraphs.
In the classic examples, one usually tries to decompose the edge set of the complete graph $G=K_n$ into copies of a fixed graph $H$ (solved by Wilson's Theorem, which is a special case of the existence of designs~\cite{glock2023existence,Keevash}), into trees of certain size (e.g. the conjectures of Ringel and Gyárfás, the former recently being solved~\cite{keevash2020ringels,MontgomeryPokrovskiySudakov}), or into $2$-regular graphs (e.g.~the Oberwolfach problem~\cite{Oberwolfach}).
There are also far-reaching results for more general families of graphs (e.g.~\cite{ABCT}, see also the survey~\cite{GlockKuhnOsthus}).

For many of these problems, the ``immediate'' divisibility conditions are conjectured/proved to be sufficient for the perfect decomposition in question when $n$ is sufficiently large~\cite[Conjecture~1.4]{Oberwolfach}. For the square of a Hamilton cycle, this means that the degrees should be divisible by $4$, i.e. $n \equiv 1 \pmod 4$.

In terms of small values, for all primes up to $p\leq 61$, $p \equiv1 \pmod 4$,  Theorem~\ref{thm:packHC2} applies. It is not difficult to check that it is impossible to partition $K_9$ into two copies of $C_9^2$. Computer-aided searches of the second author together with Sebastian Lüderssen and Fabien Nießen~\cite{Silas+Friend} found examples of decompositions for $n=21, 25, 33, 45, 49, 57, 65, 69$, which leaves $n=73$ as the smallest unresolved case.
Furthermore, in Appendix~\ref{sec:appendix_packing}, we also describe a construction of an optimal packing of squares of Hamilton cycles into $K_{105}$, representing a relatively large non-prime order $n$.
We believe this to be sufficient evidence to conjecture that $n=9$ is the only exception.

\begin{conjecture}
    \label{conj:HC2}    
    For any $n \equiv 1 \pmod 4$, $n \not= 9$, there is a decomposition of $E(K_n)$ into copies of $C_n^2$.
\end{conjecture}

This is of course a strengthening of the special case of the conjecture from~\cite{Oberwolfach} for squares of Hamilton cycles, yet its main goal is not really to identify all the exceptional small $n$ for which a decomposition does not exist. It is rather to inspire novel ways of constructing the decomposition explicitly. 
The anticipated proof technique for~\cite[Conjecture~1.4]{Oberwolfach} is the powerful method of iterative absorption (see~\cite{glock2023existence}). By its nature, this will only serve as a proof of existence. 
Even though for an infinite sequence of $n$ an explicit construction is superior to proof of existence, explicit constructions for design-type questions are extremely rare to come by.
Finding an explicit construction for a decomposition of $E(K_n)$ into copies of $C_n^2$, even just for a new infinite sequence of $n\equiv 1 \pmod{4}$ (besides the one in Theorem~\ref{thm:packHC2}), would be highly interesting. 

\paragraph{Dimension $d\geq 3$.}
    We have determined precisely for every $n$ the maximal diameter $H_s(n,2)$ of a $2$-dimensional simplicial complex on $n$ vertices.   
    Using the ``partition/cut/glue'' approach mentioned in the introduction we are able to prove the following improved lower bound for $H_s(n,d)$ for $d \ge 3$.

\begin{theorem}\label{thm:main_large_d}
    For any $d\ge 3$, we have  
    \[\left(\frac{1}{d} - n^{-1/2+o(1)}   \right) \binom{n}{d} \le H_s(n,d).\]
\end{theorem}

This improves for arbitrary $d\geq 3$ the error term of Bohman and Newman~\cite{bohman2022complexes} from inverse logarithmic to inverse polynomial.
The strategy is to first partition ``most'' edges of the complete $d$-uniform hypergraph into a limited number of so-called {\em desired cycles} (essentially squares of tight cycles).
We then cut these into \emph{desired paths} and due to our flexibility about where we cut, we obtain a good variety of vertex sequences at the ends.
Finally, we glue these pieces up into the promised \emph{desired trail} sequence, which gives a simplicial $d$-complex with large diameter.
The detailed proof of this argument appears in a separate manuscript \cite{PRS_diameter}, where we also conjecture that the trivial upper bound is tight for every $d\geq 3$ and large enough $n$.

    \begin{conjecture}
    \label{conj:exact}
        For every $d\geq 3$, we have $$H_s(n,d) = \left\lfloor \frac 1d \cdot \binom{n}{d} - \frac{d+1}{d} \right\rfloor$$ for all $n\geq n_0(d)$.    
    \end{conjecture}
    
\paragraph{History of the problem and pseudo-manifolds.}
The question of the maximum possible diameter of simplicial $d$-complexes originated from the Hirsch conjecture, which was disproved by Santos \cite{HirschCounterexample}. The polynomial Hirsch conjecture however remains open: Is there a polynomial $p(n,d)$ such that any $d$-dimensional polytope with $n$ facets has diameter at most $p(n,d)$? If the answer is ``no'', then this also implies that there is no polynomial pivot rule for the simplex algorithm. 
To approach this question,
Santos~\cite{santosderzweite} considered the more general concept of
simplicial $d$-complexes instead of polytopes and established a superpolynomial lower bound.
Criado and Santos also studied the problem for the more restrictive concept of 
\emph{$d$-dimensional pseudo-manifolds without boundary}. These are simplicial $d$-complexes in which every ridge (sets of size $d-1$) is contained in exactly two facets (sets of size $d$). This is in contrast to a shortest path in simplicial $d$-complexes, where this number is {\em at most} two.
Analogous to $H_s(n,d)$ let $H_{pm}(n,d)$ be the maximum diameter of a $d$-dimensional pseudo-manifold.
Criado and Santos~\cite{criado2017maximum} showed that any $(d-1)$-dimensional simplicial complex with $n$ vertices can be turned into a $(d-1)$-dimensional pseudo-manifold without boundary with $2n$ vertices and diameter at least two larger.
Again, Criado and Newman~\cite{criado2021randomized} and subsequently Bohman and Newman~\cite{bohman2022complexes} improved the constructions and currently the best known bounds from~\cite{bohman2022complexes} are
\[ \left( \frac{2}{(d+1)(d+1)!} - (\log n)^{-\eps} \right) n^d \le H_{pm}(n,d) \le \frac{2n^d}{(d+1)(d+1)!} \, . \]
In our language, for $d=2$ the object of interest is a ``trail-cube'' or the third power of paths.
It would be interesting to see if our approach also can be generalized to this setting, but already Theorem~\ref{thm:packHC2} does not extend naturally to the third power of a Hamilton cycle.

\paragraph{Relation to $2$-radius sequences and harmonious colourings.}

While finalizing our manuscript we learned about a closely related set of problems that goes back even further than the definition of $H_{s}(n,d)$.
For two integers $n \ge k$ a sequence of elements from $\Z /n\Z$ is called a \emph{$k$-radius sequence} if every pair from $\Z /n\Z$ appears at least once at distance at most $k$ in that sequence. 
Let $f_k(n)$ be the length of a shortest $k$-radius sequence.
Note that an easy double-counting lower bound is $f_k(n) \ge \lceil \tfrac 1k \binom{n}{2} + \tfrac{k+1}{2} \rceil$.
This concept was introduced by Jaromczyk and Lonc~\cite{jaromczyk2005sequences} 20 years ago, motivated by a problem from transmission of data sets, and has since then been studied from different angles (see~\cite{bondy2016constructing} and the references therein), for example, the general asymptotically optimal bound $f_k(n) = \tfrac{1}{2k}n^2 + \mathcal O (n^{1+\eps})$ was obtained by Jarmoczyk, Lonc, and Truszczy\'{n}ski~\cite{jaromczyk2012constructions}.

A $2$-radius sequence in which every pair appears exactly once at distance two gives an optimal simplicial complex because as a \labels{} sequence together with the all zero \layout{} it gives a good $(\labels,\layout)$ pair.
This is only possible when $n \equiv 1 \pmod 4$ and if it exists implies that $f_2(n)=H_s(n,2)-3$ in this case.
Notably, Bondy, Lonc, and Rzążewski~\cite{bondy2016constructing} precisely determined $f_2(n)$ for all but few residues classes of $n$ (see~\cite[Corollary~4.11]{bondy2016constructing,bondy2016constructing_erratum}) and thereby also determine $H_s(n,2)$ for any $n \equiv 1 \pmod 4$ such that $n \not=9$, $n \not\equiv 21 \pmod{24}$, and $n \not\equiv 7665 \pmod{8760}$.
The other way round our generating sequence for $n \equiv 5 \pmod 8$ does not use turns (c.f.~Lemma~\ref{lem:4k+1}) and, hence, our result determines $f_2(n)$ for these $n$, which solves the case $n \equiv 21 \pmod{24}$ (and with~\cite[Lemma~2.1]{bondy2016constructing} also the open cases $n \equiv 18,19,20 \pmod{24}$ of~\cite[Corollary~4.11]{bondy2016constructing}).
Interestingly, they use the concept of a $2$-additive sequence, which very closely resembles our definition of a generating sequence, and find these to construct the optimal $2$-radius sequences.
When $n \not\equiv 1 \pmod 4$ then there is no direct connection between the two concepts, because an optimal $2$-additive sequence will repeat some pairs within distance two and an optimal simplicial complex will need to have turns.
There also is a notion of a related $2$-perfect sequence, which was known in the literature before, see e.g.~\cite{bange1983sequentially}.

Another concept related to our $d>2$ case was investigated by Dębski, Lonc, and Rzążewski~\cite{dkebski2018achromatic}.
They study \emph{harmonious} colorings of a $k$-uniform hypergraph $H$, which is a coloring of the vertices of $H$ such that each edge receives $k$ different colors and every $k$-subset of colors is used at most once.
The \emph{harmonious number} of a $k$-uniform hypergraph $H$ is the minimum number of colors in a harmonious coloring of $H$.
Note that the maximum number $N$ such that the harmonious chromatic number of the desired path on $N$ vertices is at most $n$ is a lower bound for $H_s(n,d)+d+1$.
Indeed, identifying vertices with the same color gives a desired trail of length $N$ with only $n$ vertices.

A class of hypergraphs $\mathcal{H}$ is called \emph{fragmentable} if for every $\eps>0$, there is a $c>0$ such that for any $H \in \mathcal{H}$ without isolated vertices there is a set $S \subseteq V(H)$ of size $\eps v(H)$ such that $H-S$ only has components of size at most $c$.

\begin{theorem}[Dębski, Lonc, and Rzążewski~\cite{dkebski2018achromatic}]
    Let $\mathcal{H}$ be a fragmentable class of $k$-uniform hypergraphs of bounded maximum degree. Then for every $\eps>0$, there exists an $m_0$ such that for any $H \in \mathcal{H}$ with $m \ge m_0$ edges the harmonious number is at most $(1+\eps) \sqrt[k]{k!m}$.
\end{theorem}

This result is asymptotically optimal and can be applied to the family of desired paths, which is a fragmentable family of $d$-uniform hypergraphs, where each member with $N$ vertices has $dN-\binom{d+1}{2}$ edges and maximum degree bounded by $d^2$.
Hence, the theorem gives the upper bound $(1+\eps) \sqrt[d]{d!dN-d!\binom{d+1}{2}}$ for the harmonious number and, with the observation above, implies $H_{s}(n,d) \ge (1-\eps) \tfrac 1d \binom{n}{d}$.
This asymptotically determines $H_{s}(n,d)$, just as Bohman and Newman~\cite{bohman2022complexes} do, but without an explicit error term.

\paragraph{Note.}
In very recent ongoing work~\cite{Glock_diameter} together with Stefan Glock, we confirm Conjecture~\ref{conj:exact}. The proof is using iterative absorption, hence the construction is not explicit.

\subsection*{Acknowledgements}
Funded by the Deutsche Forschungsgemeinschaft (DFG, German Research Foundation) under Germany's Excellence Strategy – The Berlin Mathematics Research Center MATH+ (EXC-2046/1, project ID: 390685689).

We are grateful to G\"unter Rote for drawing our attention to the problem and its design theoretic nature and to Zbigniew Lonc for introducing us $2$-radius sequences and harmonious colorings.
We thank the HPC service of the Freie Universit\"at Berlin for the computation time provided (10.17169/refubium-26754). 

\subsection*{Data Availablility Statement}
No datasets were generated or analyzed during the current study.

\bibliographystyle{amsplain}
\bibliography{bib}

\appendix

\section{Detailed Checks of The Constructions}
\subsection{Check for Figure~\ref{fig:4n+4}}\label{sec: Check for 4n+4}
Here we check that the triangles shown in Figure~\ref{fig:4n+4} cover exactly those $20k+8$ edges that have residue 1 or 2 or contain $a$, $b$, or $c$.

 We start by checking that the number of edges is correct. The construction in Figure~\ref{fig:4n+4} has the following edges:
\begin{itemize}
    \item There is one edge $\lbrace 4k-2,4k-1\rbrace$ before the rotation.
    \item The path of the rotation contains $1+4k-7+9=4k+3$ vertices.
    As each of these vertices is incident to $a$ and this path contains $4k+2$ edges, the rotation covers $2(4k+3)-1=8k+5$ edges.
    \item Between the rotation and the zig-zag there is one edge $\{ 4k,0\}$.
    \item The zig-zag path contains $2k-3+2k-3=4k-6$ vertices.
    As each of these vertices is incident to $b$ and $c$ and the path contains $4k-7$ edges, the zig-zag covers $3(4k-6)-1=12k-19$ edges.
    \item After the zig-zag, there are $20$ additional edges.
\end{itemize}
We can see that the construction indeed covers exactly $1+(8k+5)+1+(12k-19)+20=20k+8$ edges.

Finally, we have to check that all edges with residue 1 or 2 as well as all edges containing $a$, $b$ or $c$ appear in the construction at least once.
\begin{itemize}
    \item Every edge containing $a$ appears in the rotation.
    \item The zig-zag contains all edges $\lbrace b, x\rbrace$ with $x\in\lbrace 0,1,\dots,4k-7\rbrace$. The edge $\lbrace b,4k\rbrace$ is in the rotation and the edges $\lbrace b,x\rbrace$ with $4k-6\le x\le 4k-1$ appear after the zig-zag. The edge $\lbrace b,c\rbrace$ is the last edge of the construction and we can conclude that all edges containing $b$ appear in the construction.
    \item The zig-zag contains all edges $\lbrace c,x\rbrace$ with $x\in\lbrace 0,1,\dots,4k-7\rbrace$. The edges $\lbrace c,x\rbrace $ with $4k-6\le x\le 4k$ appear after the zig-zag. Hence, the construction contains all edges containing $c$.
    \item Regarding edges with residue 1, the rotation contains all of them except for $\lbrace 4k-7,4k-8\rbrace$, $\lbrace 4k-5,4k-6\rbrace$, $\lbrace 4k-4,4k-5\rbrace$, $\lbrace 4k-2,4k-3\rbrace$, $\lbrace 4k-1,4k-2\rbrace$, $\lbrace 4k,4k-1\rbrace$ and $\lbrace 0,4k\rbrace$.  $\lbrace 4k-7,4k-8\rbrace$ appears in the zig-zag, $\lbrace4k-1,4k-2\rbrace$ appears before the rotation, $\lbrace 0,4k\rbrace$ appears between the rotation and the zig-zag. All other edges appear after the zig-zag.
    \item Regarding edges with residue 2, the zig-zag path contains all of them except $\lbrace x,x+2\rbrace$ with $4k-8\le x\le 4k$~\footnote{In the inequality $4k-8\le x\le 4k$ is an integer, while in $\{x,x+2 \}$ it is an element of $\Z/n'\Z$.}. The edges $\lbrace 4k-8,4k-6\rbrace$, $\lbrace 4k-7,4k-5\rbrace$, $\lbrace 4k-5,4k-3\rbrace$, $\lbrace 4k-2,4k\rbrace$ and $\lbrace 4k-1,0\rbrace$ are in the rotation. The edges $\lbrace 4k-6,4k-4\rbrace$,  $\lbrace 4k-4,4k-2\rbrace$, $\lbrace 4k-3,4k-1\rbrace$ and $\lbrace 4k,1\rbrace$ appear after the zig-zag.
\end{itemize}

\subsection{Check for Figure~\ref{fig:n=4k+3}}\label{sec: Check for 4n+3}
Here we check that the triangles shown in Figure~\ref{fig:n=4k+3} cover exactly those $16k+6$ edges that have residue 1 or 2 or contain $a$ or $b$ or are $\lbrace 0,13\rbrace$.

We start by confirming that the number of edges is correct. The construction in Figure~\ref{fig:n=4k+3} has the following edges:

\begin{itemize}
    \item The path of the first rotation contains  $2+(2k-6)+(2k-3)+3=4k-4$ vertices. Hence, the rotation covers $2(4k-4)-1=8k-9$ edges. However, the edge $\lbrace 0, 7\rbrace$ is considered to belong to $\towriteornottowrite{\C}{C}$; hence it is not counted here. Thus, this only covers $8k-10$ edges
    \item Between the rotations there are 35 edges.
    \item The path of the second rotation contains  $2+(4k-12)=4k-10$ vertices. Hence, the rotation covers $2(4k-10)-1=8k-21$ edges.
    \item After the second rotation, there are 2 additional edges.
\end{itemize}
We can see that the construction indeed covers exactly $(8k-10)+35+(8k-21)+2=16k+6$ edges.

To finish the proof, we have to check that all edges with residue 1 or 2 as well as all edges containing $a$ and $b$ and the edge $\lbrace 0, 13\rbrace$ appear in the construction in Figure~\ref{fig:n=4k+3} at least once.

\begin{itemize}
    \item The edge $\lbrace 0, 13\rbrace$ appears in the first rotation.
    \item Regarding the edges containing $a$, the first rotation covers all of them except for edges of the form $\lbrace a, x\rbrace$ with $1\le x\le 6$. Those appear between the two rotations.
    \item Regarding the edges containing $b$, the second rotation contains all except for edges of the form $\lbrace b,x\rbrace$ with $x=a$ or $1\le x\le 11$. The edges $\{ b,a \}$ and $\{ b,11 \}$ appear in the first rotation, while the rest appears between the two rotations
    \item Regarding the edges of residue 1, all of them appear in the second rotation except for the edges $\lbrace x, x+1\rbrace$ with $0\le x\le 11$ and $x=4k-1$. The edges $\lbrace 4k-1, 4k\rbrace$ and $\lbrace8,9\rbrace$ appear in the first rotation and the edge $\lbrace 11, 12\rbrace$ appears after the second rotation. All other edges appear between the two rotations.
    \item Regarding the edges of residue 2, they all appear in the first rotation except for the edges $\lbrace x,x+2\rbrace$ with $0\le x\le 7$ and $x\in\lbrace 11,4k-1,4k\rbrace$. The edge $\lbrace 11,13\rbrace$ appears after the second rotation and the edge $\lbrace 4k-1,0\rbrace$ appears in the second rotation. Every other edge appears between the two rotations.
\end{itemize}

\subsection{Check for Figure~\ref{fig:n=4k+2}}\label{sec: Check for 4n+2 first side}
Here we check that the triangles shown in Figure~\ref{fig:n=4k+2} cover exactly those $16k+6$ edges stated in Lemma~\ref{lem: n=4k+2 first side}.

We start by confirming that the number of edges is correct.

\begin{itemize}
    \item The path of the first rotation contains  $3+(2k-9)+(2k-8)+3=4k-11$ vertices. Hence, the rotation covers $2(4k-11)-1=8k-23$ edges. However, the edge $\lbrace 6, 17\rbrace$ is considered to belong to $\towriteornottowrite{\C_k}{C}$; hence it is not counted here. Thus, this only covers $8k-24$ edges
    \item Between the two rotations there are 13 edges.
    \item The path of the second rotation contains  $2+(4k-18)+2=4k-14$ vertices. Hence, the rotation covers $2(4k-14)+-1=8k-29$ edges.
    \item After the second rotation, there are 46 additional edges.
\end{itemize}
We can see that the construction indeed covers exactly $(8k-24)+13+(8k-29)+46=16k+6$ edges.

To finish the check, we have to certify that all edges stated in the lemma appear somewhere in the construction of Figure~\ref{fig:n=4k+2}.

\begin{itemize}
    \item The edge $\lbrace 0,17\rbrace$ appears in the first rotation.
    \item Regarding the edges containing $a$, they all appear in the first rotation except for $\lbrace a,x\rbrace$ with $2\le x\le 5$ or $7\le x\le 12$ or $14\le x\le 16$. The last three edges appear between the two rotations. All others appear after the second rotation.
    \item Regarding the edges containing $b$, the second rotation contains all of them except for $\lbrace b,x\rbrace$ with $x=1$ or $3\le x\le 16$. The edges $\lbrace b,1\rbrace$ and $\lbrace b,13\rbrace$ appear in the first rotation. The edges $\lbrace b,14\rbrace$,$\lbrace b,15\rbrace$ and $\lbrace b,16\rbrace$ appear between the two rotations. All other edges appear after the second rotation.
    \item Regarding the edges corresponding to residue 1, they all appear in the second rotation except for $\lbrace x,x+1\rbrace$ with $0\le x\le 16$ or $x=18$. For $0\le x\le 13$, this edge appears after the second rotation. For $14\le x\le 16$, this edge appears between the two rotations. The edge $\lbrace 18,19\rbrace$ appears in the first rotation.
    \item Regarding the edges corresponding to residue 2, they all appear in the first rotation except for $\lbrace x,x+2\rbrace$ with $0\le x\le 17$. The edges $\lbrace 0,2\rbrace$ and $\lbrace 17,19\rbrace$ appear in the second rotation. For $x\le 12$, the edge appears after the second rotation. For $13\le x\le 16$, the edge appears between the two rotations.
\end{itemize}

\subsection{Check for Figure~\ref{fig:n=4k+3 second half}}\label{sec: Check for 4k+2 second side}
Here we check for both constructions shown in Figure~\ref{fig:n=4k+3 second half}, that they cover exactly those $20k+14$ edges that are listed in Lemma~\ref{lem: n=4k+2 second side}. 

We start with the construction for even $k$ and count the number of edges it covers.
\begin{itemize}
    \item There is one edge $\{ 6,c\}$ before the first zig-zag not counting the edge $\lbrace0,6\rbrace$.
    \item The first zig-zag path contains $(k+1)+k+1=2k+2$ vertices. Hence, the zig-zag covers $3(2k+2)-1=6k+5$ edges.
    \item There are 3 edges between the first and the second zig-zag.
    \item The second zig-zag path contains $(k-2)+(k-4)=2k-6$ vertices. Hence, the zig-zag covers $3(2k-6)-1=6k-19$ edges.
    \item There are three edges between the second zig-zag and the rotation.
    \item The rotation has a path of length $3+(\frac{k}{2}-1)+\frac{k}{2}+\frac{k}{2}+(\frac{k}{2}-1)+(\frac{k}{2}-2)+\frac{k}{2}+\frac{k}{2}+(\frac{k}{2}+1)+1=4k+1$ vertices. Hence, the rotation covers $2(4k+1)-1=8k+1$ edges.
    \item After the rotation there are 20 edges.
\end{itemize}

We can see that the construction indeed covers 
$1+(6k+5)+3+(6k-19)+3+(8k+1)+20=20k+14$ edges.

To finish the proof for even $k$, we have to check that all edges stated Lemma~\ref{lem: n=4k+2 second side} appear somewhere in the construction of Figure~\ref{fig:n=4k+3 second half}.

\begin{itemize}
    \item The edges $\lbrace x,y\rbrace$ with $x\in\lbrace a,b\rbrace$ and $y\in\lbrace c,d,e\rbrace$ all appear at the very end of the construction. There, we also find the three edges $\lbrace c,d\rbrace$, $\lbrace d,e\rbrace$ and $\lbrace e,c\rbrace$.
    \item Regarding edges containing $c$, the edges $\lbrace c,x\rbrace$ with $x\equiv 0$ or 3 mod 4 all appear in the first zig-zag. The other edges all appear in the second zig-zag except for $\lbrace c,2\rbrace$, $\lbrace c,6\rbrace$, and $\lbrace c,4k-x\rbrace$ for $x \in \{ 3,7,11,15 \}$.
    The edge $\lbrace c,2\rbrace$ is in the first zig-zag and $\lbrace c,6\rbrace$ appears before the first zig-zag.
    The edges $\{ c,4k-3 \}$ and $\{ c,4k-11 \}$ appear between the second zig-zag and the rotation and the other two edges appear after the rotation.
    \item Regarding edges containing $d$, the edges $\lbrace d,x\rbrace$ with $d\equiv 0$ or 3 mod 4 all appear in the first zig-zag. The other edges all appear in the second zig-zag except for $\lbrace d,2\rbrace$, $\lbrace d,6\rbrace$, and $\lbrace d,4k-x\rbrace$ for $x \in \{ 3,7,11,15 \}$.
    The first two are between the two zig-zags and the other four appear after the rotation.
    \item Regarding edges containing $e$, they all appear in the rotation.
    \item Regarding edges of residue 4, the edges $\lbrace x,x+4\rbrace$ with $x\equiv 0$ or 3 mod 4 all appear in the first zig-zag~\footnote{Here $x$ is treated as an integer in $x\equiv 0$ or 3 mod 4, but as an element of $\Z/n'\Z$ in $\lbrace x,x+4\rbrace$.}. All others appear in the second zig-zag except for $\lbrace 2,6\rbrace$, $\lbrace 6,10\rbrace$, $\{4k-19,4k-15 \}$, $\{4k-15,4k-11 \}$, $\lbrace 4k-11,4k-7\rbrace$, $\lbrace 4k-7,4k-3\rbrace$ and $\lbrace 4k-3,0\rbrace$. The first two appear between the two zig-zags, $\{4k-19,4k-15 \}$ appears in the rotation, and the last three appear after the rotation.
    \item Regarding edges of residue 8, they all appear in the rotation except for $\lbrace 4k-19, 4k-11\rbrace$ and $\lbrace 4k-7, 4k-15\rbrace$. The former appears after the second zig-zag, and the latter after the rotation.
\end{itemize}
This completes the proof for even $k$.

For odd $k$, we start again by checking that the number of edges is correct. As before, there must be $20k+14$ edges.

\begin{itemize}
    \item There are $3$ edges before the first zig-zag not counting the edge $\lbrace 0,6\rbrace$.
    \item The first zig-zag path contains $k+k+2=2k+2$ vertices. Hence, the zig-zag covers $3(2k+2)-1=6k+5$ edges.
    \item Between the first zig-zag and the second rotation, there are $17$ edges.
    \item The path of the rotation contains $2+(k-5)+k=2k-3$ vertices. Hence, the rotation covers $2(2k-3)-1=4k-7$ edges.
    \item Between the first rotation and the second zig-zag, there are $3$ edges.
    \item The second zig-zag path contains $\frac{4k-4}{8}+\frac{4k-12}{8}+\frac{4k-20}{8}+\frac{4k+4}{8}=2k-4$ vertices. Hence, the zig-zag covers $3(2k-4)-1=6k-13$ edges.
    \item There is actually one edge $\{ e,4k-3 \}$ contained in the second zig-zag and also the second rotation.
    \item The path of the second rotation contains $1+\frac{4k+4}{8}+\frac{4k-4}{8}+2+\frac{4k+4}{8}+\frac{4k+4}{8}=2k+4$ vertices. Hence, the rotation covers $2(2k+4)-1=4k+7$ edges.
\end{itemize}
Therefore, the construction indeed covers $3+(6k+5)+17+(4k-7)+3+(6k-13)-1+(4k+7)=20k+14$ edges.

One last time, we are checking that all edges listed in Lemma~\ref{lem: n=4k+2 second side} appear somewhere in the construction.
\begin{itemize}
    \item The edges $\lbrace x,y\rbrace$ with $x\in\lbrace a,b\rbrace$ and $y\in\lbrace c,d,e\rbrace$ all appear between the first zig-zag and the first rotation. There, we also find the three edges $\lbrace c,d\rbrace$, $\lbrace d,e\rbrace$, and $\lbrace e,c\rbrace$.
    \item Regarding edges containing $c$, all edges of the form $\lbrace c,x\rbrace$ with $x\equiv 0$ or 3 mod 4 appear in the first zig-zag except $\{ c,0\}$, which is immediately before it. In the first rotation, we find all the other edges except for $\lbrace c,2\rbrace$, $\lbrace c,10\rbrace$, and $\lbrace c,6\rbrace$. The first two appear in the first zig-zag, and the last one appears before the first zig-zag.
    \item Regarding edges containing $d$, all edges of the form $\lbrace d,x\rbrace$ with $x\equiv 0$ or 3 mod 4 appear in the first zig-zag except for $\lbrace d,0\rbrace$, which is between the first rotation and the second zig-zag.
    The others appear in the second zig-zag except for $\lbrace d,2\rbrace$, $\lbrace d,10\rbrace$, $\lbrace d,6\rbrace$ and $\lbrace d,14\rbrace$. The first two appear in the first zig-zag, the others between this zig-zag and the first rotation.
    \item Regarding edges containing $e$, all edges of the form $\lbrace e,x\rbrace$ with $x\equiv 1$ or 2 mod 4 appear in the second zig-zag or immediately after it except for $\lbrace e,2\rbrace$, $\lbrace e,10\rbrace$, $\lbrace e,6\rbrace$ and $\lbrace e,14\rbrace$. The edges $\lbrace e,2\rbrace$, and $\lbrace e,6\rbrace$ appear in the second rotation, the edges $\lbrace e,10\rbrace$ and $\lbrace e,14\rbrace$ appear between the first zig-zag and the first rotation. The edges $\lbrace e,x\rbrace$ with $x\equiv 0$ or 3 mod 4 all appear in the second rotation.
    \item Regarding edges of residue 4, the edges $\lbrace x,x+4\rbrace$ with $x\equiv 0$ or 3 mod 4 all appear in the first zig-zag except $\lbrace 0,4\rbrace$ which appears right before it. All others appear in the first rotation except for $\lbrace 2,6\rbrace$, $\lbrace 6,10\rbrace$, $\lbrace 10,14\rbrace$, $\lbrace 18,22\rbrace$, and $\lbrace 4k-3,0\rbrace$.
    The edge $\lbrace 2,6\rbrace$ appears in the second rotation, $\lbrace 6,10\rbrace$, $\lbrace 10,14\rbrace$ before the first rotation, and $\lbrace 18,22\rbrace$ in the second zig-zag.
    \item Regarding edges of residue 8, the edges $\lbrace x,x+8\rbrace$ with $x\equiv 1$ or 2 mod 4 appear in the second zig-zag except for $\lbrace 4k-7,0\rbrace$, $\lbrace 4k-3,4\rbrace$, $\lbrace 2,10\rbrace$, $\lbrace 10,18\rbrace$, $\lbrace 6,14\rbrace$ and $\lbrace 14,22\rbrace$. The first two edges appear immediately before and after the second zig-zag, $\lbrace 2,10\rbrace$ is in the first zig-zag, $\lbrace 10,18\rbrace$ and $\lbrace 6,14\rbrace$ are before the first rotation and  $\lbrace 14,22\rbrace$ is in the first rotation. The edges $\lbrace x,x+8\rbrace$ with $x\equiv 0$ or 3 mod 4 all appear in the second rotation.
\end{itemize}

\section{Optimal Simplicial Complexes for Small $n$}\label{sec:appendix}
For small $n$, we can use a Brute Force Search to find optimal simplicial complexes.
Table~\ref{tab:small complexes} describes the $(\layout{},\labels{})$ pair of an optimal simplicial complex\towriteornottowrite{}{es} for all $n$ for which our construction above did not work.

\begin{sidewaystable}
    \begin{tabular}{|c|p{\linewidth}|}
        \hline
        $n$ & optimal simplicial complex \\\hline
          3  & [], [0,1,2] \\\hline
          4 & [0], [0,1,2,3] \\\hline
          5 & [0,0,0], [0,1,2,3,4,0] \\\hline 
          6 & [0,0,0,0,1], [0,1,2,3,4,0,5,1]\\\hline 
          7 & [0,0,0,1,1,0,0,1,1],[0,1,2,3,4,5,0,6,4,1,5,2]\\\hline 
          8 & [0,0,0,0,0,0,0,1,0,1,1,0], [0,1,2,3,4,0,5,6,2,7,3,5,1,4,6]\\\hline 
          9 & [0,0,0,0,0,0,0,0,1,0,0,1,0,0,0,1], [0,1,2,3,4,0,5,6,1,7,4,8,6,2,5,7,3,8,0]\\\hline
          10 & [0,0,0,0,0,0,0,0,1,0,0,1,1,0,1,0,1,0,1,0,1],[0,1,2,3,4,0,5,6,2,7,8,9,0,7,4,1,6,8,3,9,5,7,1,8] \\\hline 
          11 & [0,0,0,0,0,0,0,0,0,0,1,1,0,1,1,0,0,0,1,1,1,0,0,1,1,0],[0,1,2,3,4,0,5,6,1,7,4,8,6,9,10,1,8,0,10,7,2,6,3,5,8,2,9,7,3] \\\hline 
          12 & [0,0,0,0,0,0,0,0,0,0,0,0,0,0,0,0,0,0,0,0,0,0,1,0,0,1,1,0,1,1,1],[0,1,2,3,4,0,5,6,1,4,7,8,0,9,10,1,8,11,2,5,7,9,6,2,10,11,4,9,3,7,10,5,8,6]\\\hline 
          13 & [0,0,0,0,0,0,0,0,0,0,0,0,0,0,0,0,0,0,0,0,0,0,0,0,0,0,0,0,0,0,0,0,0,0,0,0,0],[0,1,3,7,8,10,1,2,4,8,9,11,2,3,5,9,10,12,3,4,6,10,11,0,4,5,7,11,12,1,5,6,8,12,0,2,6,7,9,0]\\\hline 
          14 & [0,0,0,0,0,0,0,0,0,0,0,0,0,0,0,0,0,0,0,0,0,0,1,0,0,1,0,1,1,1,0,1,0,0,1,0,1,0,1,0,1,0,1,0] \\ & [0,1,2,3,4,5,6,0,3,7,8,1,4,9,0,8,10,2,5,7,9,6,10,1,11,4,7,12,13,0,11,5,1,13,8,11,6,12,2,13,9,11,3,12,10,13,7] \\\hline
          15 & [1,0,1,0,1,0,1,0,1,0,1,0,1,0,1,0,1,0,1,0,1,0,1,0,1,1,1,1,1,1,1,1,1,1,0,0,1,1,0,1,1,1,1,0,1,1,0,1,1,1,1] \\ & [0,1,3,4,6,7,9,10,12,0,2,3,5,6,8,9,11,12,1,2,4,5,7,8,13,0,5,10,2,9,3,11,4,12,6,1,14,9,4,10,3,8,1,7,11,0,6,14,5,12,7,2,8,0]\\\hline 
          16 & [1,0,1,0,1,0,1,0,1,0,1,0,1,0,1,0,1,0,1,0,1,0,1,0,1,1,1,1,1,1,1,1,1,1,1,0,0,0,0,1,1,1,0,0,0,0,0,0,1,1,0,0,1,0,0,0,0,0]\\ & [0,1,3,4,6,7,9,10,12,0,2,3,5,6,8,9,11,12,1,2,4,5,7,8,13,0,5,10,1,6,11,3,9,2,14,4,12,15,5,14,11,7,12,6,0,15,8,2,10,7,15,3,14,8,1,9,15,4,11,10,0]\\\hline
          18 & [1,0,1,0,1,0,1,0,1,0,1,0,1,0,1,0,1,0,1,0,1,0,1,0,1,0,1,1,1,1,1,1,1,0,1,1,1,1,1,1,1,1,1,1,1,1,1,0,1,1,0,1,1,0,1,1,0,1,0,0,1,1,0,1,1,0,0,1,0,0,1,1,0,1,0]\\ & [0,1,3,4,6,7,9,10,12,0,2,3,5,6,8,9,11,12,1,2,4,5,7,8,10,11,0,5,13,6,14,7,15,8,16,17,1,6,11,3,9,2,7,12,4,10,13,14,5,15,11,17,12,6,15,2,10,1,9,14,8,2,13,3,17,10,\\&14,15,12,13,4,11,14,17,9,13,7,1] \\\hline 
          19 & [0,0,0,0,0,0,0,0,0,0,0,0,0,0,0,0,0,0,0,0,0,0,0,0,0,0,0,0,0,0,0,0,0,0,0,0,0,0,0,0,0,0,0,0,0,0,0,1,1,1,1,1,1,0,1,1,0,1,1,0,1,1,0,1,1,1,1,1,1,1,1,0,0,1,1,0,0,1,1,1,0,0,1,1] \\ & [0,3,7,15,1,5,13,16,3,11,14,1,9,12,16,7,10,14,5,8,12,3,6,10,1,4,8,16,2,6,14,0,4,12,15,2,10,13,0,8,11,15,6,9,13,4,7,11,2,17,5,9,10,12,13,18,14,15,17,0,16,18,1,17,14, 12,18,10,8,6,7,9,3,4,18,6,5,2,3,1,0,9,18,8,7,5,0]\\\hline
          22 &[1,0,1,0,1,0,1,0,1,0,1,0,1,0,1,0,1,0,1,0,1,0,1,0,1,0,1,0,1,0,1,0,0,0,1,0,1,1,1,1,1,1,1,1,1,1,1,1,1,1,1,1,1,0,0,0,1,1,1,1,1,1,0,1,1,1,1,1,1,0,1,1,0,1,1,1,0,1,1,1,0,0,1,1,1,1, 1,1,0,0,0,0,0,0,0,0,1,0,0,0,0,1,0,0,1,1,1,1,1,1,0,0,0,1] \\ & [0,3,8,11,16,2,7,10,15,1,6,9,14,0,5,8,13,16,4,7,12,15,3,6,11,14,2,5,10,13,1,4,9,12,0,17,18,7,14,4,11,1,8,15,5,12,2,9,16,6,13,3,10,19,20,21,6,2,19,4,5,7,8,10,17,9,11, 12,14,20,0,21,1,2,4,3,17,15,20,11,13,0,16,19,1,3,20,12,21,14,17,1,5,20,9,13,21,17,4,8,20,17,7,11,19,21,15,14,13,12,8,9,7,3,21,5,19]\\\hline 26        &[0,0,0,0,0,0,0,0,0,0,0,0,0,0,0,0,0,0,0,0,0,0,0,0,0,0,0,0,0,0,0,0,0,0,0,0,0,0,0,0,0,0,0,0,0,0,0,0,0,0,0,0,0,0,0,0,0,0,0,0,0,0,0,1,0,1,1,1,1,1,1,1,1,1,1,1,1,1,1,1,1,1,1,1,1,1, 0,1,1,1,1,1,1,1,1,1,1,1,1,1,1,1,1,1,1,1,1,1,0,1,1,1,1,1,1,0,1,0,0,0,0,0,1,1,1,0,0,0,0,1,0,0,1,1,0,0,0,1,1,0,0,0,1,1,1,1,0,0,0,0,0,0,1,0,0,1,0,0,0,0,1]\\&[0,4,11,19,2,9,17,0,7,15,19,5,13,17,3,11,15,1,9,13,20,7,11,18,5,9,16,3,7,14,1,5,12,20,3,10,18,1,8,16,20,6,14,18,4,12,16,2,10,14,0,8,12,19,6,10,17,4,8,15,2,6,13,0,21, 22,5,10,15,20,4,9,14,19,3,8,13,18,2,7,12,17,1,6,11,16,23,24,25,1,2,3,0,16,13,4,5,6,7,8,9,10,11,12,14,15,17,18,19,20,21,1,0,2,23,17,25,18,0,19,23,1,3,4,6,21,5,25,2, 4,21,7,23,5,8,6,9,25,7,10,8,11,21,9,12,15,25,23,10,12,13,25,14,16,17,21,19,16,18,15,23,13,11,14]\\\hline
          30 & [0,0,0,0,0,0,0,0,0,0,0,0,0,0,0,0,0,0,0,0,0,0,0,0,0,0,0,0,0,0,0,0,0,0,0,0,0,0,0,0,0,0,0,0,0,0,0,0,0,0,0,0,0,0,0,0,0,0,0,0,0,0,0,0,0,0,0,0,0,0,0,0,0,0,0,0,0,0,0,0,0,0,0,0,0,0, 0,0,0,0,0,0,0,0,0,0,0,0,0,0,1,0,1,1,1,1,1,1,1,1,1,1,1,1,1,1,1,1,1,1,1,1,1,1,1,1,1,0,0,1,1,1,1,1,1,1,1,1,1,1,1,1,1,1,1,1,1,1,1,1,1,1,1,1,0,1,1,1,1,1,1,1,1,1,1,1,1,1,1,1,1,1,1, 1,0,0,0,0,0,0,0,0,0,0,0,0,0,1,0,1,0,0,1,0,0,0,0,0,0,0,0,0,0,1,1,1,0,1,0,0,0,0,0,0,0,0]\\&[0,3,10,15,21,24,6,11,17,20,2,7,13,16,23,3,9,12,19,24,5,8,15,20,1,4,11,16,22,0,7,12,18,21,3,8,14,17,24,4,10,13,20,0,6,9,16,21,2,5,12,17,23,1,8,13,19,22,4,9,15,18,0, 5,11,14,21,1,7,10,17,22,3,6,13,18,24,2,9,14,20,23,5,10,16,19,1,6,12,15,22,2,8,11,18,23,4,7,14,19,0,25,26,8,16,24,7,15,23,6,14,22,5,13,21,4,12,20,3,11,19,2,10,18,1, 9,17,27,28,29,2,4,5,1,0,17,13,9,7,3,19,15,11,10,6,8,12,14,16,18,20,21,22,23,24,25,27,1,2,3,4,6,5,7,8,9,10,12,11,13,14,15,16,17,18,22,20,27,19,21,17,28,15,13,27,12, 16,28,20,24,22,1,3,28,5,27,9,11,28,7,6,27,2,23,0,21,25,28,19,18,14,27,10,8,28,4,0,27,24]\\
          \hline
    \end{tabular}
    \caption{Optimal simplicial complexes for small $n$}
    \label{tab:small complexes}
\end{sidewaystable}

It turns out that we can match the trivial upper bound for all $n$ except for $n=6$. Even though one can check this by hand, there is a short argument for why $H_s(6,2)$ cannot be $6$:

\begin{lemma}
    $H_s(6,2)<6$.
\end{lemma}
\begin{proof}
Suppose, $H_s(6,2)=6=\frac{1}{2}\binom{6}{2}-\frac{3}{2}$. Then, there is a simplicial $2$-complex on $\lbrace 0,1,2,3,4,5\rbrace$ such that each edge appears exactly once. In particular, every vertex appears in exactly 5 edges. 

\begin{figure}[H]
    \centering
    \begin{tikzpicture}[scale=1]
        \begin{scope}[shift={(0,0)}]
        \foreach \x in {0,1,6,5} {
        \begin{scope}[shift={(\x,0)}]
        \draw (0,0) -- (0:1);
        \draw[] (0:1) -- (60:1) -- (0,0);
        \end{scope}
        }
        \draw (0.5,0.86603) -- (1.5,0.86603);
        \draw (5.5,0.86603) -- (6.5,0.86603);
        \textat{0,-.2}{$a$}
        \textat{7,-.2}{$d$}
        \textat{0.5,1.06603}{$b$}
        \textat{6.5,1.06603}{$c$}
        \end{scope}
\end{tikzpicture}
    \caption{The first and last three triangles of a simplicial 2-complex always look like this}
    \label{fig:n=6 is bad}
\end{figure}
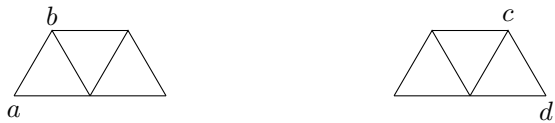

Let $a$, $b$, $c$, and $d$ as shown in Figure \ref{fig:n=6 is bad}. Since $b$ has three incident edges, the other time the label of $b$ appears in the complex, it must be incident to exactly two edges. But only $a$ and $d$ are incident to exactly two edges. As $a$ and $b$ cannot have the same label, the labels of $b$ and $d$ are the same. Analogously, the labels of $a$ and $c$ are the same. But now, the edges $ab$ and $cd$ have the same labels as endpoints, which gives a contradiction.
\end{proof}

\section{Packing of Squares of Hamilton Cycles in $K_n$ for $n=105$.}

Similarly to the generating sequences one can build squares of Hamilton cycles out of short sequences: Given a sequence $(a_0,\dots,a_\ell)$ we want to build $\ell$ squares of Hamilton cycles by starting one cycle at each of $0,1,\dots,\ell$ and always increasing the label periodically by $a_0,\dots, a_\ell$ modulo $n$.
For instance, the sequence $(19,10,4)$ creates the cycle

(0, 19, 29, 33, 52, 62, 66, 85, 95, 99, 13, 23, 27, 46, 56, 60, 79, 89, 93, 7, 17, 21, 40, 50, 54, 73, 83, 87, 1, 11, 15, 34, 44, 48, 67, 77, 81, 100, 5, 9, 28, 38, 42, 61, 71, 75, 94, 104, 3, 22, 32, 36, 55, 65, 69, 88, 98, 102, 16, 26, 30, 49, 59, 63, 82, 92, 96, 10, 20, 24, 43, 53, 57, 76, 86, 90, 4, 14, 18, 37, 47, 51, 70, 80, 84, 103, 8, 12, 31, 41, 45, 64, 74, 78, 97, 2, 6, 25, 35, 39, 58, 68, 72, 91, 101).

By starting another cycle with label $1$ instead of $0$, we get the cycle 

(1, 20, 30, 34, 53, 63, 67, 86, 96, 100, 14, 24, 28, 47, 57, 61, 80, 90, 94, 8, 18, 22, 41, 51, 55, 74, 84, 88, 2, 12, 16, 35, 45, 49, 68, 78, 82, 101, 6, 10, 29, 39, 43, 62, 72, 76, 95, 0, 4, 23, 33, 37, 56, 66, 70, 89, 99, 103, 17, 27, 31, 50, 60, 64, 83, 93, 97, 11, 21, 25, 44, 54, 58, 77, 87, 91, 5, 15, 19, 38, 48, 52, 71, 81, 85, 104, 9, 13, 32, 42, 46, 65, 75, 79, 98, 3, 7, 26, 36, 40, 59, 69, 73, 92, 102).

And by starting with label $2$, we get

(2, 21, 31, 35, 54, 64, 68, 87, 97, 101, 15, 25, 29, 48, 58, 62, 81, 91, 95, 9, 19, 23, 42, 52, 56, 75, 85, 89, 3, 13, 17, 36, 46, 50, 69, 79, 83, 102, 7, 11, 30, 40, 44, 63, 73, 77, 96, 1, 5, 24, 34, 38, 57, 67, 71, 90, 100, 104, 18, 28, 32, 51, 61, 65, 84, 94, 98, 12, 22, 26, 45, 55, 59, 78, 88, 92, 6, 16, 20, 39, 49, 53, 72, 82, 86, 0, 10, 14, 33, 43, 47, 66, 76, 80, 99, 4, 8, 27, 37, 41, 60, 70, 74, 93, 103).

These three cycles cover all edges with residue in $(19\upper{29}10\upper{14}4\upper{23})$.

Be aware that not all sequences work in that way. However, the following sequences define 26 squares of Hamilton cycles where every edge is covered exactly once:

\begin{align*}
    &(19\upper{29}10\upper{14} 4\upper{23})\\
&(-40\upper{22} -43\upper{-38} 5\upper{-35})\\
&(-41\upper{-13} 28\upper{-52} 25\upper{-16})\\
&(-48\upper{-42} 6\upper{27} 21\upper{3} -18\upper{-44} -26\upper{31})\\
&(-17\upper{-9} 8\upper{-46} 51\upper{-7} 47\upper{-2} -49\upper{39})\\
&(-30\upper{-50} -20\upper{-32} -12\upper{24} 36\upper{37} 1\upper{-33} -34\upper{11} 45\upper{15}) \, . 
\end{align*}
\label{sec:appendix_packing}

\end{document}